\documentclass[11pt,a4paper,reqno]{amsart}
\usepackage{cmap}
\usepackage[T1]{fontenc}
\usepackage{graphicx}
\usepackage{subcaption}
\usepackage{setspace,mathrsfs,amsthm,amsmath,amssymb,amsfonts,lipsum,appendix,mathtools,cite,tabularx,fontenc,bbm,caption}

\usepackage[shortlabels]{enumitem}
\makeatletter
\newtheorem*{rep@theorem}{\rep@title}
\newcommand{\newreptheorem}[2]{%
	\newenvironment{rep#1}[1]{%
		\def\rep@title{#2 \ref{##1}}%
		\begin{rep@theorem}}%
		{\end{rep@theorem}}}
\makeatother

\newtheorem{theorem}{Theorem}[section]
\newtheorem{lemma}[theorem]{Lemma}

\newtheorem{corollary}[theorem]{Corollary}
\newtheorem{proposition}[theorem]{Proposition}
\theoremstyle{definition}
\newtheorem{remark}[theorem]{Remark}
\newtheorem{definition}[theorem]{Definition}

\newtheorem{example}[theorem]{Example}

\newtheorem{conjecture}[theorem]{Conjecture}
\usepackage[nointegrals]{wasysym}
\usepackage[misc]{ifsym}
\usepackage[dvipsnames]{xcolor}
\definecolor{ao}{rgb}{0.0, 0.5, 0.0}
\definecolor{lasallegreen}{rgb}{0.0, 0.3, 0.0}
\usepackage{tikz}
\usepackage{hyperref}
\hypersetup{colorlinks,allcolors=black}
\usepackage[margin=0.75in]{geometry}
\usetikzlibrary{calc,patterns,angles,quotes}

\usepackage{wrapfig}

\let\oldnorm\norm
\def\norm{\@ifstar{\oldnorm}{\oldnorm*}}

\newcommand{\pa} {\partial}

\newcommand{\De} {\Delta}

\newcommand{\Om} {\Omega}
\newcommand{\la} {\lambda}

\newcommand{\si} {\sigma}

\newcommand{\Lom} {\mathcal{L}}

\newcommand{\R}{{\mathbb R}}

\newcommand{\N}{{\mathbb N}}

\usepackage[hyperpageref]{backref}

\usepackage{orcidlink}

\usepackage{a4wide}
\usepackage{csquotes}

\usepackage{amsfonts}
\usepackage{amssymb,amsthm}
\usepackage{amsmath}
\allowdisplaybreaks
\usepackage{mathtools}
\mathtoolsset{showonlyrefs}

\numberwithin{equation}{section}

\usepackage {setspace}

\usepackage{enumitem}

\setlist{nosep}

\begin{document}
\singlespacing

\title[Reverse Faber-Krahn inequalities for the Logarithmic potential operator]{Reverse Faber-Krahn inequalities for the Logarithmic potential operator}

\author[T.V.~Anoop]{T.V.~Anoop}
\author[Jiya Rose Johnson]{Jiya Rose Johnson}

\address[T.V.~Anoop]{\newline\indent
	Department of Mathematics,
	Indian Institute of Technology Madras, 
	\newline\indent
	Chennai 36, India
	\newline\indent
	\orcidlink{0000-0002-2470-9140} 0000-0002-2470-9140 
}
\email{anoop@iitm.ac.in}
\address[Jiya Rose Johnson]{\newline\indent
	Department of Mathematics,
	Indian Institute of Technology Madras, 
	\newline\indent
	Chennai 36, India
}
\email{jiyarosejohnson@gmail.com}









\begin{abstract} For a bounded open set $\Omega \subset \R^2,$ we consider the largest eigenvalue $\tau_1(\Om)$ of the Logarithmic potential operator $\Lom$. If $diam(\Omega)\le 1$, we prove reverse Faber-Krahn type inequalities for  $\tau_1(\Om)$  under polarization and Schwarz symmetrization. Further, we establish the monotonicity of $\tau_1(\Om\setminus\mathcal{O})$ with respect to certain translations and rotations of the obstacle $\mathcal{O}$ within $\Omega$. The analogous results are also stated for the largest eigenvalue of the Riesz potential operator. Furthermore, we investigate properties of the smallest eigenvalue $\Tilde{\tau}_1(\Om)$ for a domain whose transfinite diameter is greater than 1. Finally, we characterize the eigenvalues of $\Lom$ on $B_R$, including the  $\Tilde{\tau}_1(B_R)$ when $R>1$.
\end{abstract}

\maketitle
\definecolor{lblack}{gray}{0.3}
\definecolor{mygray}{gray}{0.9}
\definecolor{vlgray}{gray}{0.96}
\definecolor{medgray}{gray}{0.8}
\definecolor{dgray}{gray}{0.7}
\begin{quote}	
	\setcounter{tocdepth}{1}
	\tableofcontents
	\addtocontents{toc}{\vspace*{0ex}}
\end{quote}
\section{Introduction}
\par In 1877, Lord Rayleigh, in his famous book \textit{"The theory of sound\cite{rayleigh}"} conjectured that \textit{"among all planar domains of a given area, disc minimizes the first Dirichlet eigenvalue"}. This conjecture was independently proved by Faber  \cite{faber}  and Krahn\cite{krahn1}. Later, Krahn \cite{krahn2}  extended this result for the domains in higher dimensions. For a bounded domain $\Omega$ in $\R^N$, let $\lambda_1(\Omega)$ be denote the first Dirichlet eigenvalue of Laplacian on $\Omega$. 
Then, Faber-Krahn inequality states that
\begin{equation}\label{FK}
    \la_1(\Omega^*)\leq \la_1(\Omega),
\end{equation}
where $\Omega^*$ is the open ball centred at the origin with the same measure as $\Omega$. The above inequality is known as Faber-Krahn inequality. 
For $N=2$, Faber and Krahn also proved that equality holds in \eqref{FK} only when $\Omega$ is a disc. For $N\ge 3$,  Kawohl \cite{kawohl1985} and Kesavan \cite{kes-fk-unique} extended this result to smooth domains,  and Daners-Kennedy \cite{fk-unique-2007}  for general bounded domains. A simple proof \eqref{FK} can be given using Schwarz symmetrization and P\'{o}lya-Szeg\"{o}  inequality\cite{polyaszego}. 
It states that, for $1\leq p<\infty$ and for a non-negative function $u\in W_0^{1,p}(\Omega),$
\begin{equation}\label{polyaszego}
    \int_{\Omega^*} |\nabla u^*|^p dx\leq \int_\Omega |\nabla u|^p dx,
\end{equation}
where $u^*$ is the Schwarz symmetrization of the function $u$. For various proofs of \eqref{polyaszego}, see \cite{bandle_c},\cite{keijohilden},\cite{liebpolyaszego}. 

 \par In \cite{Anoop-Ashok2023}, authors proved a Faber-Krhan type inequality for the first Dirichlet eigenvalue of $p$-Laplacian on $\Omega$ and its Polarization. The Polarization is one of the simplest rearrangements on $\R^N$. Polarization for sets was introduced by Wolontis\cite{wolontis}, and later Ahlfors\cite{ahlfors}  advanced this rearrangement for function transformations. For more details on polarization, refer to \cite{brockpol}, \cite{Burchard2009ASC}, \cite{dubinin}, and \cite{a_y_solynin}.  
 \begin{definition}{\textbf{(Polarization).}}
    A polarizer $H$ is an open affine halfspace in $\mathbb{R}^N.$ Let us denote the reflection with respect to the boundary $\partial H$ by $\sigma_H$. For a set $\Omega$ in $\R^N,$  the polarization $P_H(\Omega)$ with respect to $H$ is defined as below:
    \begin{equation*}
        P_H(\Omega)=[(\Omega\cup \sigma_H(\Omega))\cap H]\cup [\Omega\cap\sigma_H(\Omega)]
    \end{equation*}
\end{definition}
 A polarization $P_H$ preserves the Lebesgue measure and the set inclusions, and it takes open sets to open sets and closed sets to closed sets.
 \\
 \begin{minipage}[t]{0.475\textwidth} 
      \begin{center}
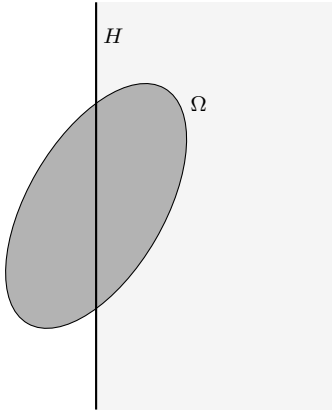

      \captionsetup{type=figure}
      \begin{tikzpicture}[scale=0.9]
         \fill[fill=vlgray] (0,-3) rectangle (3.5,3);
         \def\firstellipse{(0,0) ellipse [x radius=2, y radius=1, rotate=60]};
         \fill[dgray] \firstellipse;
         \draw[black] \firstellipse;
         
         \draw[-, thick] (0,-3) -- (0,3);
         \begin{scriptsize}    
          
             \draw (0.25,2.5) node {$H$};
             \draw (1.5,1.5) node {$\Omega$};
         \end{scriptsize}
      \end{tikzpicture}
      \captionof{figure}{ The grey region is $\Omega$}
  \end{center}
  \end{minipage}
  \begin{minipage}[t]{0.525\textwidth} 
  \begin{center}
  \captionsetup{type=figure}
      \begin{tikzpicture}[scale=0.9]
         \fill[fill=vlgray] (0,-3) rectangle (3.5,3);
         \def\firstellipse{(0,0) ellipse [x radius=2, y radius=1, rotate=60]};
         \draw[black] \firstellipse;
         \fill[dgray] \firstellipse;
         \def\secellipse{(0,0) ellipse [x radius=2, y radius=1, rotate=120]};
         \fill[dgray] \secellipse;
         \draw[black] \secellipse;
         \begin{scope}
         \clip \firstellipse;
         \fill[darkgray] \secellipse;
         \end{scope}  
         \def\rect{(0,-3) rectangle (3.5,3)}
         \begin{scope}
         \clip \rect;
         \fill[darkgray] \secellipse;
         \draw[black] \secellipse;
         \end{scope}  
         \begin{scope}
         \clip \rect;
         \fill[darkgray] \firstellipse;
        
         \end{scope}  
        \draw[black] \firstellipse;
         \draw[black] \secellipse;
           
         \draw[-, thick] (0,-3) -- (0,3);
         \begin{scriptsize}    
           
             \draw (0.25,2.5) node {$H$};
              \draw (1.75,1.5) node {$P_H(\Omega$)};
         \end{scriptsize}
      \end{tikzpicture}
      \begin{center}
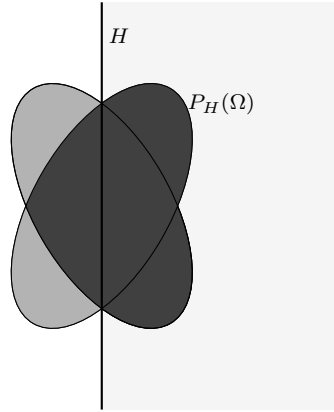

          \captionof{figure}{ The dark region is $P_H(\Omega)$}
      \end{center}
      
  \end{center}
 \end{minipage}
\begin{definition}{\textbf{(Polarization of a function).}}
    For a measurable function $u:\mathbb{R}^N\rightarrow \mathbb{R}$, the polarization $P_H(u)$ with respect to $H$ is defined as 
   \begin{equation*}
       P_H(u)(x)=\begin{cases}
           \max\{u(x),u(\sigma_H(x)\},& \text{   for }x\in H\\
           \min\{u(x),u(\sigma_H(x))\},& \text{   for }x\in \mathbb{R}^N\setminus H
       \end{cases}
   \end{equation*}
   Let $u:\Omega\rightarrow \mathbb{R}$ and let $\Tilde{u}$ be its zero extension to $\mathbb{R}^N$. The polarization $P_H(u)$ is defined as the restriction of $P_H(\Tilde{u})$ to $P_H(\Omega)$.
\end{definition}

One can notice that, if $u=\chi_\Om$, then $P_H(u)=\chi_{P_H(\Om)}$. Moreover, for $1\leq p<\infty$ and for a nonnegative function $u\in W^{1,p}_0 (\Omega)$, we have $P_H(u)\in W^{1,p}_0 (P_H(\Om))$ and  the following equality in analogy to $\eqref{polyaszego}$ (see \cite[Corollary 5.1]{brockpol}):
\begin{equation}\label{polyaszego_pol}
\int_{P_H(\Omega)} |\nabla P_H(u)|^p dx = \int_\Omega |\nabla u|^p dx.
\end{equation}
The above equality, together with the variational characterization of the first Dirichlet eigenvalue of the $p-$Laplacian, easily gives the following inequality:
\begin{equation}\label{ashokresult}
    \la_1(P_H(\Omega))\leq \la_1(\Omega).
\end{equation}
 In \cite{Anoop-Ashok2023}, it is proved that the equality is attained in equation \eqref{ashokresult} if only if $P_H(\Omega)=\Omega$ or $P_H(\Omega)=\sigma_H(\Omega)$. 

\par In this article, we study the above inequality and the equality cases for the eigenvalues of the Logarithmic potential operator.  
\begin{definition}{\bf (Logarithmic potential operator).}
 For a bounded open set $\Om$ in $\R^2$, the Logarithmic potential operator $\Lom$, on $L^2(\Omega)$ is defined as: 
\begin{align}
    \Lom u(x)&= \frac{1}{2\pi}\int_\Omega \log\frac{1}{|x-y|}u(y)dy.
\end{align}
\end{definition}


\noindent The operator $\Lom$ is self-adjoint and compact on $L^2(\Omega)$. By the spectral theorem, the set of eigenvalues of $\Lom$ is countably infinite with zero as its only limit point. The corresponding eigenfunctions are orthogonal in $L^2(\Om)$. Clearly, each eigenpair $(\tau, u)$ satisfy:
\begin{equation}\label{var_char}
    \frac{1}{2\pi}\iint\limits_{\Om\;\Om}\log\frac{1}{|x-y|} u(x)\phi(y)dxdy=\tau \int_\Omega u(x)\phi(x)dx,\,\; \forall\, \phi\in L^2(\Omega).
\end{equation}
For $u\in L^2(\Om)$, we define
\begin{equation*}
    E(u):=\frac{1}{2\pi}\iint\limits_{\Om\;\Om}\log\frac{1}{|x-y|} u(x)u(y)dxdy.
\end{equation*}
The largest eigenvalue of $\Lom$ is denoted by $\tau_1(\Om)$ and has the following variational characterization:
\begin{equation}\label{pevpos}
    \tau_1(\Om):=\sup \left\{ E(u):u\in L^2(\Omega),\|u\|_{L^2(\Omega)}=1\right\}.
\end{equation}

\noindent Our first theorem establishes a reverse Faber-Krahn type inequality for $\tau_1$ under Polarization. To state the theorem, we need the following definition.
\begin{definition}\textbf{(Almost equal).}
    Let $A$ and $B$ be two measurable subsets of $\R^N$. We say that $A$ is almost equal to $B$ (denoted by $A\cong B$) if $|A\triangle B|=0$, where $\triangle$ is the symmetric difference of the sets.
\end{definition}
\begin{theorem}\label{fklog}
    Let $\Omega\subset \R^2$ be a bounded open set and let $H$ be a polarizer such that an eigenfunction corresponding to $\tau_1(\Omega)$  is positive on $\Om\setminus \sigma_H(\Om)$. Then,
    \begin{equation}\label{pollog}
            \tau_1(\Omega)\leq \tau_1(P_H(\Omega)).
        \end{equation}
         The equality holds in $\eqref{pollog}$ only if $P_H(\Omega)\cong\Omega$ or $P_H(\Omega)\cong\sigma_H(\Omega)$.
\end{theorem}
If  $diam(\Omega)\leq 1$, we establish that  $\tau_1(\Om)$ admits an eigenfunction  which is positive on entire $\Om$ (see Proposition \ref{prop-non-neg}). Now, as a consequence of the above theorem, we obtain the following corollary for open sets $\Om$ with $diam(\Om)\leq 1$:
\begin{corollary}
    Let $\Omega\subset \R^2$ be a bounded open set with $diam(\Om)\leq 1$. Then, for any polarizer $H$, \eqref{pollog} holds. Furthermore, the equality holds in $\eqref{pollog}$ only if $P_H(\Omega)\cong\Omega$ or $P_H(\Omega)\cong\sigma_H(\Omega)$.
\end{corollary}
\begin{remark}\label{Rem:Suragan}
\begin{enumerate}[(i)]
    \item     In \cite[Lemma 3.1]{suragan2016}, authors claimed that if $\Lom$ is a positive operator then every eigenfunction corresponding to $\tau_1(\Om)$ has constant sign. This is not true even for the unit ball $B_1$ as there are sign-changing eigenfunctions corresponding  $\tau_1(B_1)$ of $\Lom$, see \cite[Theorem 3.1]{anderson}. In the proof this lemma, authors asserts that if $u$ is a sign-changing eigenfunction corresponding to $\tau_1(\Om)$, then there exists $x_0\in\Om$ and a neighbourhood $B(x_0,r)$ such that \begin{equation}
        |u(x)u(y)|-u(x)u(y)>0\quad \forall \, x,y\in B(x_0,r).
    \end{equation}
    Their assertion is incorrect as the above inequality fails if $u(x)$ and $u(y)$ have the same sign.

    \item We conjecture that if $\Lom$ is a positive operator on $L^2(\Om)$, then at least one of the eigenfunctions corresponding to $\tau_1(\Om)$ is positive. 
    \item If $\Lom$ is not a positive operator, it must have a unique negative eigenvalue\cite[Theorem 2 $\&$ 3]{troutman1967} and the corresponding eigenfunction will be positive\cite[Theorem 1]{troutman1967}. Consequently, the eigenfunction corresponding to $\tau_1(\Om)$ will be sign-changing. In this case, also, we anticipate 
    \begin{equation*}
         \tau_1(\Omega)\leq \tau_1(P_H(\Omega)).
    \end{equation*}

\end{enumerate}
\end{remark}

In \cite[Theorem 2.5]{suragan2016}, authors considered bounded domains $\Om$ for which $\Lom$ is a positive operator on $L^2(\Om)$, and they claimed   that the following reverse Faber-Krahn inequality holds:
\begin{equation}\label{fkaliterlog}
\tau_1(\Omega) \leq \tau_1(\Omega^*)
\end{equation}
 However, their proof is wrong in multiple places. Firstly, their proof uses an incorrect lemma (Lemma 3.1, see Remark \ref{Rem:Suragan}) for the existence of a positive eigenfunction corresponding to $\tau_1(\Om)$. Secondly, their proof requires a Riesz-type inequality(\cite[Lemma 5.4]{suragan2016} ) with triple integrals, and its proof is wrong. More precisely,  the equation (5.12) of \cite[Lemma 5.4]{suragan2016} is incorrect since $h_2(R)$ becomes negative for large $R$. As a consequence of Theorem \ref{fklog}, we get the reverse Faber-Krahn inequality for $\tau_1$ under Schwarz symmetrization.  
\begin{theorem}\label{faber-unique-log}
Let $\Omega \subset \mathbb{R}^2$ be a bounded open set such that 
$\tau_1(\Omega)$ possesses a positive eigenfunction. Then \eqref{fkaliterlog} holds.
Equality holds in \eqref{fkaliterlog} only if $\Omega \cong \Omega^*$ up to a translation.
\end{theorem}


Next, we study the behaviour of $\tau_1(\Omega)$  with respect to certain variations of $\Omega.$ For $0<r<R$ and $e_1=(1,0)$, we set  $$\Omega_t:=B_R(0)\setminus \overline{B_r(te_1)},$$ where $B_\rho(x)$ is the open ball centered at $x$ with radius $\rho$  in $\R^2$ and $\overline{B_\rho(x)}$ is the corresponding closed ball centred at $x$ with radius $\rho$. There is considerable interest in studying the behavior of the first eigenvalue of $\Om_t$  with respect to $t$ for various operators. In \cite{kesavan2003, Harrell}, authors established that the first Dirichlet eigenvalue  $\lambda_1(\Omega_t)$ of the Laplacian   is strictly decreasing on $[0,R-r)$. This result is extended for the $p$-Laplacian in  \cite{anoop_bobkov_sarath}. The Hadamard perturbation formula is the key ingredient of proof of the monotonicity result in all the aforementioned articles. In \cite{Ashok_JDE}, authors established several monotonicity results by exploiting the strict Faber-Krahn type inequality under the Polarization. 
For such results, one has to merely express the domain variations in terms of an appropriate set of polarizations. This idea is also used in \cite[Theorem 1.1]{ashoknew-p-q} to investigate the monotonicity of the first $q$-eigenvalue of the fractional $p$-Laplace operator.
Here, we establish the monotonicity of $\tau_1(\Omega_t)$ of $\Lom$. We need some restrictions on the outer radius $R$ to ensure the existence of a positive eigenfunction.
\begin{theorem}\label{corlog}
      If $R<\frac{1}{2}$, then  $\tau_1(\Omega_t)$ is strictly increasing for $t\in [0,R-r)$.
 \end{theorem}
We also investigate the monotonicity of $\tau_1(\Om\setminus \mathcal{O})$ for other types of obstacles and variations. We refer to Section \ref{translation_rotation} for the results in these directions.

\subsection{The smallest  eigenvalue of $\Lom$}

\par In contrast to \eqref{pevpos}, let us now consider
\begin{align}\label{pevneg}
    \Tilde{\tau}_1(\Om)&:=
    \inf \left\{ E(u):u\in L^2(\Omega),\|u\|_{L^2(\Omega)}=1\right\}.
    \end{align}
 If $\Omega$ is such that $\Lom$ is a positive operator on $L^2(\Omega)$, then one can observe that $\Tilde{\tau}_1(\Om)=0$ and
 $\Tilde{\tau}_1(\Om)$ is not an eigenvalue of $\Lom$\cite[Corollary 1]{troutman1967}.  If $\Omega$ is such that $\Lom$ is not a positive operator on $L^2(\Omega)$ then  $\tilde{\tau}_1(\Om)$ becomes negative.
 In \cite[Theorem 2 $\&$ 3]{troutman1967},  Troutman established that $\tilde{\tau}_1(\Om)$  indeed an eigenvalue of  $\Lom$ and there are no other negative eigenvalue. He also characterized the positivity of $\Lom$ in terms of a geometric property known as the {\it transfinite diameter} of $\Omega$. We recall the definition and properties of transfinite diameter (see \cite[Chapter 16]{hille2002}). 
 
\subsubsection{\textbf{Transfinite diameter}}
Let $E$ be a compact set in $\R^2$. For each $n\geq 2$, let
\begin{equation*}
    \rho_n(E)= \max_{(x_1, x_2, \cdots x_n) \in E^n} \,\,  \left [\underset{1\leq i<j\leq n}{\prod} |x_i-x_j| \right]^{\frac{2}{n(n-1)}},
\end{equation*}   
be the $n^{th}$-diameter of $E$. For $n=2$, $\rho_2(E)=diam(E)$, the usual diameter of $E$. The sequence $\rho_n$ is monotonically decreasing \cite[Theorem 16.2.1]{hille2002} and the transfinite diameter of $E$(denoted by $T_{diam}(E)$) is defined as
\begin{equation}\label{limit_n_diameter}
    T_{diam}(E)=\lim_{n\rightarrow\infty}\rho_n(E).
\end{equation}
Another definition  of the transfinite diameter is given in \cite[Theorem 16.4.4]{hille2002} as $$T_{diam}(E)=e^{- V_E},$$ where
\begin{equation}\label{robinconst}
   V_E=\underset{\nu}{\inf} \left\{ \iint\limits_{E\;E} \log\frac{1}{|x-y|}d\nu(x)d\nu(y)\right\},
\end{equation}
and $\nu$ varies over all normalized Borel measures on $E$. This constant $V_E$ is known as the Robin constant of $E$ in the literature. 

\noindent For a nonempty bounded domain  $\Omega$, we define transfinite diameter as
\begin{equation}
    T_{diam}(\Om):=T_{diam}(\overline{\Om}).
\end{equation}
The transfinite diameter of a disc is its radius, whereas, for an annular region, it is the outer radius. For an ellipse, the transfinite diameter is the average of its semi-axes. The transfinite diameter for other domains can be found in \cite{Dijkstra2008NumericalAO}, \cite{Hayman1966lecture} and  \cite{thomasransford}.

\par In\cite{anderson}, the set of all eigenvalues of $\Lom$ on $B_1$ is identified in terms of zeroes of Bessel functions. For $B_R$ with $R\leq 1$, the eigenvalues are expressed in terms of zeroes of some functions involving Bessel functions\cite{suraganvolume}. We extend the method from \cite{anderson} for arbitrary radius and exhibit all eigenvalues in the Appendix \ref{Appendix}. For $R>1$, we have $T_{diam}(B_R)=R>1$ and hence $\Tilde{\tau}_1(B_R)<0$. Next, we express $\Tilde{\tau}_1(B_R), R>1$ in terms of a zero of a function involving the modified Bessel function, 
\begin{equation}\label{I_0_series}
    I_0(x)=\sum_{m=0}^\infty \frac{x^{2m}}{2^{2m}(m!)^2},
\end{equation}
that satisfies the following differential equation:
    \begin{equation}\label{modified_bessel}
        t^2 I_0''(t)+t I_0'(t)-t^2I_0(t)=0.
    \end{equation}
 
\begin{theorem}\label{radial_small_version}
    Let $B_R$ be a disc with radius $R>1$. Then, the unique negative eigenvalue $\Tilde{\tau}_1(B_R)$ and the corresponding eigenfunction(up to a constant multiple) are given by
    \begin{equation*}
        \Tilde{\tau}_1(B_R)=-\frac{R^2}{\mu_{0}(B_R)^2}, \quad \Tilde{u}(x)= I_0\left(\frac{\mu_{0}(B_R)}{R}|x|\right),
    \end{equation*}
    where $\mu_{0}(B_R)$ is the unique root of the  equation
    $I_0(t)-\log R\, t I_0'(t)=0 $. Moreover, if $R_1<R_2$, then $\mu_0(B_{R_2})<\mu_0(B_{R_1})$.
   \end{theorem}

\par The following theorem provides asymptotic estimates for $\Tilde{\tau}_1(B_R)$ for $R$ near to 1 and near to $\infty$.
\begin{theorem}\label{asymptotic_small_version}
     Let $B_R$ be a disc with radius $R>1$. Then,
        \begin{equation*}
        \Tilde{\tau}_1(B_R)\approx\begin{cases} -R^2(\log\, R)^2 
                    & \text{ when  $R$ is near } 1,\\
         -\frac{R^2 \,\log\, R}{2}     & \text{ when $R$ is near $\infty$} .
        \end{cases}
    \end{equation*}
\end{theorem}


The rest of the article is organized as follows. In Section 2, we discuss the preliminaries. The proof of the main theorems is given in Section 3. In section 4, we discuss further properties of eigenvalues, such as their monotonicity under various domain perturbations, the behaviour of  $\Tilde{\tau}_1$ as the diameter increases, and the approximation of the eigenvalues of a domain $\Om$ with eigenvalues of discs. In Section 5, we extend the isoperimetric inequalities previously proven for the Logarithmic potential operator to the case of Riesz potential operators and also include conjectures and open problems related to transfinite diameter and $\Tilde{\tau}_1$. Finally, in Appendix \ref{Appendix} and Appendix \ref{AppendixB}, we analyze the eigenvalues of $\Lom$ on $B_R$.

\section{Preliminaries}

In this section, we recall the definitions of various symmetrizations and some properties of polarization. Furthermore, we provide some preliminary results required to prove Faber-Krahn-type inequalities. We begin by recalling the definitions. 
\begin{definition}{\textbf{(Convexity in $h-$direction).}}
    A set $C\subseteq \R^N$ is said to be convex in $h- $ direction if a line segment connecting points in $C$ is parallel to the line $\R h$, then the entire line segment is contained in $C$.
\end{definition}
\begin{definition}{\textbf{(Steiner symmetrization of a measurable set $A$).}}\label{def_steiner}
    Let $A\subseteq \R^N$ be a measurable set and $S$ be an affine hyperplane in $\R^N$. The set $A$ is said to be Steiner symmetric with respect to $S$ if $A$ is symmetric under reflection with respect to $S$ and $A$ is convex in the orthogonal direction of $S$.
\end{definition}
\begin{definition}{\textbf{(Foliated Schwarz symmetrization of a measurable set $A$).}}
    Let $a\in\R^N,\eta\in S^{N-1}$ and consider the ray $a+\R^+\eta$. A measurable set $A\subset \R^N$ is said to be foliated Schwarz symmetric concerning $a+\R^+\eta$ if the following condition holds for each $r>0$.
    \begin{equation*}
        A\cap \partial B_r(a)=B_\rho(a+r\eta)\cap\partial B_r(a),
    \end{equation*}
    where $\rho$ is chosen  such that 
    \begin{equation*}
        \mathcal{H}^{N-1}(B_\rho(a+r\eta)\cap\partial B_r(a))=\mathcal{H}^{N-1}(A\cap \partial B_r(a)),
    \end{equation*}
    with $\mathcal{H}^{N-1}$ representing the $N-1$ dimensional Hausdorff measure. 
    
\end{definition}

 \begin{definition}{\textbf{(Schwarz symmetrization of a function $f$).}}
    Let $\Omega\subset \R^N$ be a bounded measurable set, and let $f:\Omega\rightarrow\R$ be a non-negative measurable function. The Schwarz symmetrization of $f$ is the function $f^*:\Omega^*\rightarrow\R$ is defined as 
    \begin{equation*}
        f^*(x)=\int_0^\infty \chi_{\{f(x)>t\}^*}(x)dt.
    \end{equation*}
\end{definition}
In the following proposition, we state some properties of polarization without proof(for the proof, see \cite[Proposition 2.3, Proposition 2.18]{Anoop-Ashok2023}).
\begin{proposition}\label{pol_properties}
    Let $\Omega\subseteq\R^N$ be an open set and $H$ be a polarizer. Then,
    \begin{enumerate}[(i)]
        \item $P_H(\Omega)\neq \Omega$ if and only if $A_H:=\sigma_H(\Omega)\cap \Omega^c\cap H$ has non-empty interior.
        \item $P_H(\Omega)\neq \sigma_H(\Omega)$ if and only if $B_H:=\Omega\cap (\sigma_H(\Omega))^c\cap H$ has non-empty interior.
        \item For a non-negative function $f:\Om\rightarrow\R$, we have 
    \begin{equation*}
        \|f\|_{L^2(\Omega)}=\|P_H(f)\|_{L^2(P_H(\Omega))}.
    \end{equation*}
    \end{enumerate}
\end{proposition}

\noindent \textbf{Riesz type inequality:} The key ingredient for proving Theorem \ref{fklog} is the Riesz-type inequality,
\begin{equation}\label{riez_inequality_pol}
        \iint\limits_{\R^N\;\R^N} f(x)f(y)K(|x-y|) dxdy \leq  \iint\limits_{\R^N\;\R^N} P_H(f)(x)P_H(f)(y)K(|x-y|) dxdy 
    \end{equation}
which is established in \cite[Lemma 2.6]{Burchard2009ASC} for non-negative measurable functions. We extend this inequality to apply to integrals involving arbitrary measurable functions. Before proceeding, we make the following proposition. Recall that, for any affine halfspace $H$, there exists a scalar $s\in \R$ and a unit vector $a\in\R^N$ such that $H=H_{s,a}:=\{x\in \R^N : x\cdot a > s\}$. We denote $\overline{x}$ for $\sigma_H(x)$.


\begin{proposition}\label{reflection_distance}
    Let $H_{s,a}$ be an open affine halfspace. If $x,y\in H_{s,a}$, then $|x-y|< |x-\overline{y}|$. On the otherhand, if $x\in H_{s,a}$ and $y\in H_{-s,-a}$, then $|x-y|> |x-\overline{y}|$.
\end{proposition}
\begin{proof}
    Notice that for $x\in\R^N$,  its reflection is given by $\overline{x}=x+2(s- (x\cdot a))a$.  Then, we have: 
    \begin{align}
        |x-\overline{y}|^2&=|x-y|^2+4((x\cdot a)-s)((y\cdot a)-s).
    \end{align} 
    Hence, if $x,y\in H_{s,a}$, then $((x\cdot a)-s)((y\cdot a)-s)>0$ implies $|x-y|< |x-\overline{y}|$. On the otherhand, if $x\in H_{s,a}$ and $y\in H_{-s,-a}$, then $((x\cdot a)-s)((y\cdot a)-s)<0$ gives $|x-y|> |x-\overline{y}|$.
\end{proof}
\noindent Consider the space
\begin{equation*}
    \mathcal{X}=\{f:f \text{ is measurable on }\R^N \text{ and } \iint\limits_{\R^N\;\R^N} \big| f(x)f(y)K(|x-y|) \big| dxdy<\infty\}.
\end{equation*}
We prove a Riesz-type inequality for functions in $\mathcal{X}$.

\begin{proposition}\label{riezpol}
    Let $K:[0,\infty)\rightarrow \R$ be a decreasing function and $H$ be a polarizer. Then, for any $f$on $\mathbb{R}^N$, such that $f,P_H(f)\in\mathcal{X}$, the inequality \eqref{riez_inequality_pol} holds. In addition, if $K$ is strictly decreasing, the equality holds in \eqref{riez_inequality_pol} only if either $ P_H(f)=f$ a.e. or $P_H(f)=f\circ \sigma_H $ a.e.
\end{proposition}
\begin{proof}
    Let 
    \begin{equation*}
        I(f)= \iint\limits_{\R^N\;\R^N} f(x)f(y)K(|x-y|) dxdy.
     \end{equation*}
\noindent By denoting $\overline{x}=\sigma_H(x)$ and using the change of variables, we get:
 \begin{equation}
     I(f)=\iint\limits_{H\;H} \big[ (f(x)f(y)+f(\overline{x})f(\overline{y}))K(|x-y|)+(f(x)f(\overline{y})+f(\overline{x})f(y))K(|x-\overline{y}|) \big] dxdy. \\
 \end{equation}
 Let 
 \begin{equation*}
     \mathcal{S}_{f}(x,y):=(f(x)f(y)+f(\overline{x})f(\overline{y}))K(|x-y|)+(f(x)f(\overline{y})+f(\overline{x})f(y))K(|x-\overline{y}|).
 \end{equation*}
  Then,
  \begin{equation}\label{I_f}
     I(f)= \iint\limits_{A\;A} \mathcal{S}_{f}(x,y) dxdy + 2 \iint\limits_{B\;A} \mathcal{S}_{f}(x,y) dxdy +\iint\limits_{B\;B} \mathcal{S}_{f}(x,y) dxdy,
 \end{equation}
 where  
 \begin{equation*}
     A=\{x\in H: f(x)\geq f(\overline{x})\}\text{ and }B=\{x\in H: f(x)< f(\overline{x})\} .
 \end{equation*}
 Notice that,
 \begin{equation}\label{x_in A_x_in B}
\begin{aligned}
  P_H(f)(x)=f(x) &\quad\text{and}\quad P_H(f)(\overline{x})=f(\overline{x}),\quad\forall\, x\in A,\\
     P_H(f)(x)=f(\overline{x}) &\quad\text{and}\quad  P_H(f)(\overline{x})=f(x),\quad\forall\, x\in B.
\end{aligned}
\end{equation}
It is easy to verify that $$\mathcal{S}_{f}(x,y)=\mathcal{S}_{P_H(f)}(x,y), \forall\, (x,y)\in A\times A \bigcup B\times B.$$ Consequently,
\begin{equation}\label{A_A}
    \iint\limits_{A\;A} \mathcal{S}_{f}(x,y) dxdy= \iint\limits_{A\;A} \mathcal{S}_{P_H(f)}(x,y) dxdy,
\end{equation}
and
\begin{equation}\label{B_B}
    \iint\limits_{B\;B} \mathcal{S}_{f}(x,y) dxdy= \iint\limits_{B\;B} \mathcal{S}_{P_H(f)}(x,y) dxdy.
\end{equation}

 
     


\noindent On the other hand, for $(x,y)\in A\times B$, using \eqref{x_in A_x_in B}, we obtain:
    \begin{equation}\label{S_PHf-S_f}
        \mathcal{S}_{P_H(f)}(x,y)-\mathcal{S}_{f}(x,y)=(f(x)-f(\overline{x}))(f(\overline{y})-f(y))(K(|x-y|)-K(|x-\overline{y}|)).
    \end{equation}
    \noindent Since  $|x-y|< |x-\overline{y}|$ (Proposition \ref{reflection_distance}) and $K$ is decreasing, we get
    \begin{equation*}
        \mathcal{S}_{P_H(f)}(x,y)-\mathcal{S}_{f}(x,y)\geq 0.
    \end{equation*}
    Therefore,
    \begin{equation}\label{B_A}
    \iint\limits_{B\;A} \mathcal{S}_{f}(x,y) dxdy\leq \iint\limits_{B\;A} \mathcal{S}_{P_H(f)}(x,y) dxdy.
\end{equation}


\noindent Now, by combining \eqref{A_A}, \eqref{B_B}and \eqref{B_A}, we conclude  $I(f)\leq I(P_H(f))$. \\ Next, we assume that $K$ is strictly decreasing and  $I(f)=I(P_H(f))$. Thus,
\begin{equation*}
    \iint\limits_{B\;A} \mathcal{S}_{f}(x,y)dxdy=\iint\limits_{B\;A} \mathcal{S}_{P_H(f)}(x,y)dxdy.
\end{equation*}
Therefore, $$\mathcal{S}_{P_H(f)}(x,y)=\mathcal{S}_{f}(x,y) \text{ a.e on } A\times B.$$ Now, let $A_1=\{x\in A: f(x)>f(\overline{x})\}.$ 
Since $K$ is strictly decreasing,  from \eqref{S_PHf-S_f} we conclude that $\mathcal{S}_{P_H(f)}(x,y)>\mathcal{S}_{f}(x,y)$, for $(x,y)\in A_1\times B$. Therefore $|A_1\times B|=0$ and hence $|A_1|=0$ or $|B|=0$. Notice that, 

 \begin{enumerate}[(i)]
    \item if $|B|=0$, then $f(x)\geq f(\overline{x})$  a.e. in $H$, and hence $P_H(f)=f$ a.e. in $\R^N,$
     \item if $|A_1|=0$, then $f(x)\leq f(\overline{x})$ for a.e. in $H$, and hence $P_H(f)=f\circ\sigma_H$  a.e. in $\R^N$.  
 \end{enumerate}
 This concludes the proof.
\end{proof}


\begin{proposition}\label{Riesz on Omega}
    Let $\Om\subseteq\R^N$ and $f\in L^2(\Om)$ be  non-negative. Then, 
    \begin{equation}\label{int_Om_int_P_H_Om}
        \iint\limits_{\Om\;\Om}\log \frac{1}{|x-y|} f(x)f(y) dxdy \leq \int\limits_{P_H(\Om)} \int\limits_{P_H(\Om)} \log \frac{1}{|x-y|}  P_H(f)(x)P_H(f)(y) dxdy.
    \end{equation}
\end{proposition}
\begin{proof}
        We extend $f$ by zero outside $\Omega$. Notice that,
    \begin{equation*}
        P_H(\Om)^c=H\setminus (\Om\cup\sigma_H(\Om)) \cup H^c\setminus (\Om\cap \sigma_H(\Om))
    \end{equation*}
    For $x\in H\setminus (\Om\cup\sigma_H(\Om))$, $f(x)=f(\overline{x})=0.$ Thus $P_H(f)(x)=0$. For $x\in H^c\setminus (\Om\cap \sigma_H(\Om))$, we have either $f(x)=0$ or $f(\overline{x})=0$. Therefore, $P_H(f)(x)=\min \{f(x),f(\overline{x})\}=0$, since $f$ is non-negative. Thus, $P_H(f)=0$ on $P_H(\Om)^c$. Since $f\in L^2(\Om)$,  it follows that both $f,P_H(f)\in\mathcal{X}$. Now, the proof follows from Proposition \ref{riezpol}.
    
\end{proof}

    


\noindent As a consequence of Proposition \ref{riezpol}, we have the following corollary.   
\begin{corollary}\label{ries_ineq_from_pol}
    Let $\Om\subseteq\mathbb{R}^N$ be a bounded domain. Then, for a non-negative measurable function $f\in L^2(\Om)$, 
     \begin{equation}\label{Riesz_Schwarz}
        \iint\limits_{\Om\;\Om} \log \frac{1}{|x-y|} f(x)f(y) dxdy \leq  \iint\limits_{\Om^*\;\Om^*} \log \frac{1}{|x-y|} f^*(x)f^*(y) dxdy.
    \end{equation}
    Moreover, the equality holds in \eqref{Riesz_Schwarz} only if $ f=f^*\circ\tau$ a.e.
\end{corollary}
The above result is also proved in \cite[Lemma 2]{CompetingSymmetries}, and Burchard gives an alternate proof via polarization in \cite[Theorem 2.10]{Burchard2009ASC}.



\noindent We now state some regularity properties of the eigenfunctions of the Logarithmic potential.
\begin{proposition}\label{prop_regularity}
    Let $(\tau,u)$ be an eigenpair of $\Lom$. Then, $u\in C(\overline{\Omega})$ and is bounded. Furthermore, $\Lom u(x)=\tau u(x)$ for all $x\in\Om$, i.e.,
    \begin{equation}\label{tau u=integral}
       \tau  u(x)=\frac{1}{2\pi}\int_\Om \log \frac{1}{|x-y|}u(y)dy.
    \end{equation}
\end{proposition}

\begin{proof}
 For any $f\in L^2(\Omega)$, it is known that $\Lom f\in C^\alpha (\Omega)$ with $0<\alpha<1$\cite[p.2]{troutman1967}. Thus, $u\in C^\alpha(\Omega)$ and hence $u\in C(\overline{\Om})$ by continuos extension. Now, from \eqref{var_char}, we obtain $\Lom u(x)=\tau u(x)$ for all $x\in\Om$.
\end{proof}

\noindent The following proposition states that polarization reduces the diameter.
\begin{proposition}\label{pol_decrs_diam}
    Let $K\subset \R^N$ be a compact set and $H$ be a polarizer. Then,
    \begin{equation*}
        diam(P_H(K))\leq diam(K).
    \end{equation*}
\end{proposition}
\begin{proof} Since $K$ is compact, $P_H(K)$ is also compact, and its diameter is attained by two points $x,y\in P_H(K)$, where $|x-y|=diam(P_H(K))$. If both $x$ and $y$ are in $K$ or in $\sigma_H(K)$, we have $|x-y|\leq diam(K)$. If one of $x$ or $y$ belongs to the intersection $K \cap \sigma_H(K)$, it reduces to one of the previous cases. Now, if $x\in [K \setminus \sigma_H(K)]\cap H$ and $y\in [\sigma_H(K) \setminus K]\cap H$, then $\overline{y} \in K$, and by Proposition \ref{reflection_distance} we have $|x-y|\leq |x-\overline{y}| \leq diam(K)$. Hence, in all cases, we have $\text{diam}(P_H(K)) \leq \text{diam}(K)$. 
\end{proof}

\section{Proof of the Main Theorems} 
In this section, we prove the main theorems stated in the introduction: Theorem \ref{fklog}, Theorem \ref{faber-unique-log}, Theorem \ref{corlog}, Theorem \ref{radial_small_version}, and Theorem \ref{asymptotic_small_version}. First, we establish the existence of a positive eigenfunction when $diam(\Om)\leq 1$, utilizing the positivity of the kernel.

\begin{proposition}\label{prop-non-neg}
    Let $\Omega\subset \R^2$ be an open set with $diam(\Omega)\leq 1$ . Then, the first eigenfunction corresponding to $\tau_1$ does not vanish in $\Om$. 
\end{proposition}
\begin{proof}
     Let $\phi_1$ be an eigenfunction corresponding to $\tau_1$ such that $\| \phi_1\|_{L^2(\Om)}=1$. Proposition \ref{prop_regularity} shows that $\phi_1$ is continuous. Moreover,
    \begin{equation}\label{contradiction}
        E( \phi_1) = \sup \left\{ E(u):u\in L^2(\Omega),\|u\|_{L^2(\Omega)}=1\right\}.
    \end{equation}
    Let
    \begin{equation*}
        \Omega^+=\{x\in\Omega: \phi_1(x)>0\},\,\,\text{  and  }
        \Omega^-=\{x\in\Omega: \phi_1(x)<0\}.
    \end{equation*}
    Suppose $\phi_1$ is sign changing, then 
    \begin{equation}\label{measnonzero}
        |\Omega^+|\neq 0 \text{ and }|\Omega^-|\neq 0.
    \end{equation}
   Moreover,
    \begin{equation}\label{signchange}
        |\phi_1(x)\phi_1(y)|>\phi_1(x)\phi_1(y),\quad \text{for all } x\in\Omega^+\, ,\, y\in\Omega^-.
    \end{equation}
    In addition, since $diam(\Om)\leq 1$, we have $|x-y|<1\text{ for } x,y\in\Om$ and hence
    \begin{equation}\label{log is positive}
        \log \frac{1}{|x-y|} > 0,\quad \forall\, x,y\in\Om.
    \end{equation}
    Now consider,
    \begin{align*}
        E (|\phi_1|) &=\frac{1}{2\pi}\iint\limits_{\Om^+\;\Om^+}\log\frac{1}{|x-y|} \phi_1(x)\phi_1(y)dxdy + \frac{1}{\pi}\iint\limits_{\Om^+\;\Om^-}\log\frac{1}{|x-y|} |\phi_1(x) \phi_1(y)|dxdy\\
        &\quad\quad\quad\quad +\frac{1}{2\pi} \iint\limits_{\Om^-\;\Om^-}\log\frac{1}{|x-y|} \phi_1(x)\phi_1(y)dxdy\\
        &> \frac{1}{2\pi}\iint\limits_{\Om^+\;\Om^+}\log\frac{1}{|x-y|} \phi_1(x)\phi_1(y)dxdy + \frac{1}{\pi} \iint\limits_{\Om^+\;\Om^-}\log\frac{1}{|x-y|} \phi_1(x)\phi_1(y)dxdy\\
        &\quad\quad\quad\quad +\frac{1}{2\pi}\iint\limits_{\Om^-\;\Om^-}\log\frac{1}{|x-y|}\phi_1(x)\phi_1(y)dxdy= E(\phi_1),
    \end{align*}
where the strict inequality follows from $\eqref{measnonzero},\eqref{signchange} $ and \eqref{log is positive}.
A contradiction to \eqref{contradiction}. Thus, $\phi_1$ can not change its sign in $\Omega$, and hence $\phi_1$ can be chosen to be non-negative.   
     Now suppose that $\phi_1(x_0)=0$ for some $x_0\in\Om$. Then,  
    \begin{equation}
        0=\tau_1 \phi_1(x_0)=\frac{1}{2\pi}\int_\Omega \log\frac{1}{|x_0-y|}\phi_1(y)dy.
    \end{equation}
    Since the integrand does not change sign, from \eqref{log is positive}, we must have $\phi_1 \equiv 0$ in $\Omega$, a contradiction as $\phi_1$ is an eigenfunction. Thus, $\phi_1$ does not vanish in $\Om$.
     
\end{proof}

\begin{remark}
    If $diam(\Om)\leq 1$, from Proposition \ref{prop-non-neg}, we can infer that every eigenfunction corresponding to $\tau_1$ has a constant sign. Moreover, for any two eigenfunctions $\phi_1,\phi_2$  corresponding to $\tau_1,$  and $x_0\in \Omega$, we can find $c\in \mathbb{R}$, such that  $\phi_1(x_0)+c\phi_2(x_0)=0.$  Again, by Proposition \ref{prop-non-neg}, we conclude that $\phi_1(x)+c\phi_2(x)=0,\forall\, x\in \Omega.$ This shows that  $\tau_1$ is simple. 
\end{remark}


\noindent Now we prove Theorem \ref{fklog}.
\\
\noindent{\bf Proof of Theorem \ref{fklog}:}
Let $\phi_1$ be the eigenfunction corresponding to $\tau_1(\Om)$, satisfying $\|\phi_1\|_{L^2(\Omega)}=1$ and $\phi_1>0$ on $\Om\setminus \sigma_H(\Om)$. Notice that,  $\|P_H(\phi_1)\|_{L^2(P_H(\Omega))}=\|\phi_1\|_{L^2(\Omega)}=1$ and $P_H(\phi_1)=0$ on $P_H(\Om)^c$. Now, by using Proposition \ref{riezpol}, we obtain:
\begin{align}\label{logpolid}
        \tau_1(\Omega)&=\frac{1}{2\pi}\iint\limits_{\Om\;\Om}\log\frac{1}{|x-y|}\phi_1(x)\phi_1(y) dxdy\nonumber\\
        &\leq \frac{1}{2\pi} \int\limits_{P_H(\Om)} \int\limits_{P_H(\Om)} \log\frac{1}{|x-y|}P_H(\phi_1)(x)P_H(\phi_1)(y)dxdy\leq \tau_1(P_H(\Omega)).\nonumber
    \end{align} 
\noindent Next, assume that $\tau_1(\Omega)=\tau_1(P_H(\Omega))$. Then, from the above inequality, we obtain
    \begin{equation*}
        \frac{1}{2\pi}\iint\limits_{\Om\;\Om} \log\frac{1}{|x-y|}\phi_1(x)\phi_1(y)dxdy\\
        = \frac{1}{2\pi}\int\limits_{P_H(\Om)} \int\limits_{P_H(\Om)}\log\frac{1}{|x-y|}P_H(\phi_1)(x)P_H(\phi_1)(y) dxdy.
    \end{equation*}
     Now, as a consequence of Proposition \ref{riezpol}, we have  either
     \begin{equation}\label{two_possible_cases}
         P_H(\phi_1)=\phi_1\text{ a.e.  in }\R^N \text{ or } P_H(\phi_1)=\phi_1\circ \sigma_H\text{ a.e.  in }\R^N.
     \end{equation}
    One can easily verify that,
    \begin{equation*}
        P_H(\phi_1)\neq \phi_1\text{ on }P_H(\Om)\,\triangle\, \Om,
    \end{equation*}
    and 
    \begin{equation*}
        P_H(\phi_1)\neq \phi_1\circ\sigma_H\text{ on }P_H(\Om)\,\triangle\, \sigma_H(\Om).
    \end{equation*}
     By \eqref{two_possible_cases}, we must have $|P_H(\Om)\,\triangle\, \Om|=0$ or $|P_H(\Om)\,\triangle\, \sigma_H(\Om)|=0$. Thus, $P_H(\Omega)\cong\Omega$ or $P_H(\Omega)\cong\sigma_H(\Omega)$.
    
\qed 
    

Next, we prove Theorem \ref{faber-unique-log} by adapting the arguments from \cite[Theorem 1.3]{Ashok-Nirjan2023}. Before proceeding, we recall the following result from \cite[Theorem 4.4]{jeanvan}.
    \begin{proposition}\label{approxschwarz}
        Let $u\in L^2(\R^N)$ be a non-negative function and $u^*$ be its Schwarz symmetrization. Then, there exists a sequence of polarizers $(H_n)_{n\in\N}$ such that
        \begin{equation*}
            P_{H_n H_{n-1}\cdots H_1}(u):=P_{H_n}(P_{H_{n-1}}(\cdots (P_{H_1}(u)))) \rightarrow u^*\text{ in }L^2(\R^N).
        \end{equation*}
    \end{proposition}
    
 \noindent{\bf Proof of Theorem \ref{faber-unique-log}:} 
 Let $\phi_1$ be the positive eigenfunction corresponding to $\tau_1(\Omega)$ such that $\|\phi_1\|_{L^2(\Omega)=1}$. Let $u$  be the zero extension of $\phi_1$. Let $H_n$ be the sequence of polarizers given by Proposition \ref{approxschwarz}. Denote $u_n=P_{H_n H_{n-1}\cdots H_1}(u)$ and  $\Om_n=P_{H_n H_{n-1}\cdots H_1}(\Om)$.  Observe that, $\text{supp}(u_n)=\overline{\Om_n}$ and $\text{supp}(u^*)=\overline{\Om^*}$. Now, by Proposition \ref{approxschwarz}, $u_n\rightarrow u^*$ in $L^2(\R^N)$ 
and hence, $\Om^*\cap \Om_n\neq \emptyset$ for all $n\geq N\in\N$. Since $diam(\Om_n)\leq diam(\Om)$(Proposition \ref{pol_decrs_diam}), there exists a radius $R>0$, such that $\Om,\Om_n\subseteq B_R,\,\,\forall \, n\in\N$, where $B_R$ denotes the ball centered at the origin with radius $R$. By repeatedly applying \eqref{riez_inequality_pol} for $u,u_1,u_2,\cdots u_{n-1}$, we get:
        \begin{align}\label{limsup}
            \iint\limits_{\R^N\;\R^N}  \log \frac{1}{|x-y|}u(x)u(y)dxdy
            &\leq \iint\limits_{\R^N\;\R^N}  \log \frac{1}{|x-y|}u_n(x)u_n(y)dxdy,\quad \forall\, n\in \N.
        \end{align}
       Notice that, up to a subsequence $u_n\rightarrow u^*$ a.e. and $\|u_n\|_{L^\infty(\R^N)}= \|u\|_{L^\infty(\Om)}<+\infty$, by Proposition \ref{prop_regularity}. Now, take the limit in the above inequality (using the dominated convergence theorem) to get:
        \begin{align*}
            \iint\limits_{B_R\;B_R}  \log \frac{1}{|x-y|}u(x)u(y)dxdy 
            &\leq \iint\limits_{B_R\;B_R}  \log \frac{1}{|x-y|}u^*(x)u^*(y)dxdy.
        \end{align*}
        Since $\text{Supp}( u)\subseteq \overline{\Omega}$ and $\text{Supp}( u^*)\subseteq \overline{\Omega^*}$,  we get the required Riesz inequality,
        \begin{equation*}
            \iint\limits_{\Om\;\Om} \log \frac{1}{|x-y|}u(x)u(y)dxdy\leq \iint\limits_{\Om^*\;\Om^*}  \log \frac{1}{|x-y|}u^*(x)u^*(y)dxdy.
        \end{equation*}  
        Furthermore, we have $\|u^*\|_{L^2(\Omega)}=\|u\|_{L^2(\Omega)}=1.$
        Now, by the variational characterization of the largest eigenvalue in $\eqref{pevpos}$, we have: 
        \begin{equation*}
            \tau_1(\Omega)\leq \tau_1(\Omega^*).
        \end{equation*}
    Now suppose $\tau_1(\Omega)= \tau_1(\Omega^*)$.  For any polarizer $H$, we have:
    \begin{equation*}
        \tau_1(\Omega)\leq \tau_1(P_H(\Omega))\text{ and }\tau_1(P_H(\Omega))\leq \tau_1((P_H(\Omega))^*)= \tau_1(\Omega^*).  
    \end{equation*}
   Therefore, $\tau_1(\Omega)= \tau_1(P_H(\Omega))$ for any polarizer  $H$. Thus, by Theorem \ref{fklog}, we have $P_H(\Omega)\cong\Omega $ or $P_H(\Omega)\cong\sigma_H(\Omega )$. From \cite[Lemma 6.3(p.1761 $\&$ p.1774)]{brockpol}, we conclude that, $\Omega\cong\Omega^*$ up to a translation.


    
\qed

     
\noindent Next, we prove the monotonicity of eigenvalues for the eccentric annular regions. Recall, $\Omega_t:=B_R(0)\setminus \overline{B_r(te_1)}$.
\\
 \begin{minipage}[t]{0.55\textwidth} 
     \begin{center}
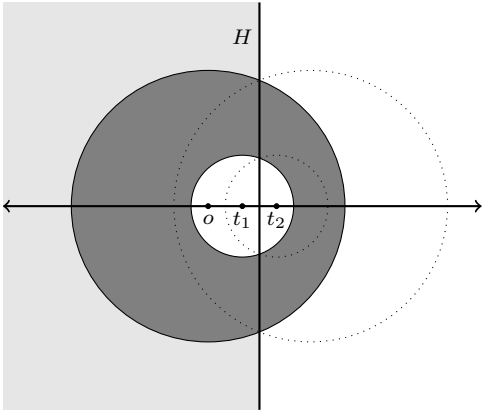

     \captionsetup{type=figure}
     \begin{tikzpicture}[scale=0.9]
        \fill[fill=mygray] (-3.001,-3.001) rectangle (0.75,3.001);
        \draw [fill=gray](0,0) circle (2cm); 
        \draw [fill=white](0.5,0) circle (0.75cm);
        \draw[color=black,dotted] (1,0) circle [radius=0.75];
         \draw[color=black,dotted] (1.5,0) circle [radius=2];
        \draw[<->, thick] (-3.001,0) -- (4,0);      
        \draw[-, thick] (0.75,-3.001) -- (0.75,3.001);
        \filldraw[black] (0,0) circle (1pt);
        \filldraw[black] (0.5,0) circle (1pt);
        \filldraw[black] (1,0) circle (1pt);
        \begin{scriptsize}    
            \draw (0,-0.2) node {$o$};
            \draw (0.5,-0.2) node {$t_1$};    
            \draw (1,-0.2) node {$t_2$}; 
            \draw (0.5,2.5) node {$H$};
        \end{scriptsize}
     \end{tikzpicture}
     \captionof{figure}{Shaded region is $\Omega_{t_1}$}
 \end{center}
 \end{minipage}
 \begin{minipage}[t]{0.45\textwidth} 
 \begin{center}
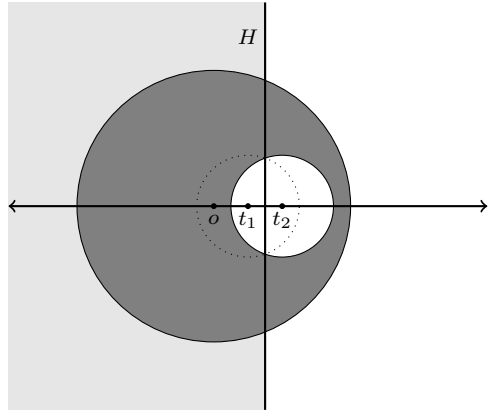

 \captionsetup{type=figure}
     \begin{tikzpicture}[scale=0.9]
        \fill[fill=mygray] (-3.001,-3.001) rectangle (0.75,3.001);
        \draw [fill=gray](0,0) circle (2cm); 
        \draw [fill=white](1,0) circle (0.75cm);
        \draw[color=black,dotted] (0.5,0) circle [radius=0.75];
        \draw[<->, thick] (-3.001,0) -- (4,0);      
        \draw[-, thick] (0.75,-3.001) -- (0.75,3.001);
        \filldraw[black] (0,0) circle (1pt);
        \filldraw[black] (0.5,0) circle (1pt);
        \filldraw[black] (1,0) circle (1pt);
        \begin{scriptsize}    
            \draw (0,-0.2) node {$o$};
            \draw (0.5,-0.2) node {$t_1$};    
            \draw (1,-0.2) node {$t_2$}; 
            \draw (0.5,2.5) node {$H$};
        \end{scriptsize}
     \end{tikzpicture}
     \captionof{figure}{$P_H(\Omega_{t_1})=\Omega_{t_2}$}
 \end{center}
 \end{minipage}
 \\
 \\
 \noindent{\bf Proof of Theorem \ref{corlog}:} Let $t_1<t_2$ and $t_1,t_2\in [0,R-r)$. Let $H=\{(x_1,x_2):x_1<\frac{t_1+t_2}{2}\}$. Observe that, $P_H(\Omega_{t_1})=\Omega_{t_2}$, $\Omega_{t_2}\ncong\Omega_{t_1}$ and $\Omega_{t_2}\ncong\sigma_H(\Omega_{t_1})$. Since $R<\frac{1}{2}$, it follows that $diam(\Om_t)<1$ for all $t$, and by Proposition \ref{prop-non-neg}, $\Om_t$ admits a positive eigenfunction. Therefore, from Theorem \ref{fklog}, we obtain 
 \begin{equation*}
     \tau_1(\Omega_{t_1})< \tau_1(P_H(\Omega_{t_1}))=\tau_1(\Omega_{t_2}).
 \end{equation*}
 This concludes the proof.\qed


\noindent Next, we will prove the theorems related to the negative eigenvalue on discs $B_R$ with $R>1$.
 \noindent{\bf Proof of Theorem \ref{radial_small_version}:}
Let $R>1$ and  let $\Tilde{u}$ be an eigenfunction corresponding to the negative eigenvalue $\Tilde{\tau}_1(B_R)$. Then, $\Tilde{u}(x)=\Phi(|x|)$  (Theorem \ref{radial}) and 
 $\Phi$ solves the modified Bessel equation(Lemma \ref{ball_lemma1}),
$$r^2\Phi''+r\Phi'-\la r^2\Phi=0,$$ with the boundary condition 
\begin{equation}
    \Phi(R)-R\log R \Phi'(R)=0,
\end{equation}
where $\Tilde{\tau}_1(B_R)=-\frac{1}{\la}.$
 Therefore,  $\Phi(r)=cI_0(\sqrt{\la} r)$, where $I_0$ is the modified Bessel function given in \eqref{I_0_series}), and $\lambda$ is such that 
\begin{equation*}
    I_0(\sqrt{\la} R)-\sqrt{\la} R\log R I_0'(\sqrt{\la} R)=0.
\end{equation*}
By Proposition \ref{bessel_prop_unique},  there exists a unique  zero of the function $I_0(t)-t\log R I_0'(t)$, say $\mu_0(B_R)$. Thus, $\sqrt{\la}R=\mu_0(B_R)$ and hence, 
    \begin{equation*}
        \Tilde{\tau}_1(B_R)=-\frac{R^2}{\mu_{0}(B_R)^2}\text{ and } \Tilde{u}(x)=\Phi(|x|)=cI_0\left(\frac{\mu_{0}(B_R)}{R}|x|\right),\quad c\in\R.
    \end{equation*}
Now, for $R_1<R_2$, again by Proposition \ref{bessel_prop_unique}, we must have:
 $$\mu_0(B_{R_2})<\mu_0(B_{R_1}).$$ 


\qed

\noindent{\bf Proof of Theorem \ref{asymptotic_small_version}:}
Recall, from Theorem \ref{radial_small_version}, the unique negative eigenvalue is $\Tilde{\tau}_1(B_R)=-\frac{R^2}{(\mu_{0}(B_R))^2}$, where $\mu_0(B_R)$ is the unique zero of $I_0(t)-\log R t I_0'(t)$. 
Let
    \begin{equation*}
        g(t)=\frac{t I_0'(t)}{I_0(t)}.
    \end{equation*}
    Then, $g$ is strictly increasing and attains all values in $[0,\infty)$ exactly once(see Proposition \ref{bessel_prop_unique}). Notice that,  $$g(\mu_0(B_R))=\frac{1}{\log\, R}.$$
    Using Mathematica, one can obtain the Taylor expansion of $g$ about zero as below:
    \begin{equation*}
        g(t)=\frac{t^2}{2}-\frac{t^4}{16}+\frac{t^6}{96}-\frac{11 t^8}{6144}+\cdots
    \end{equation*}
    By truncating the series, we obtain the approximation:
    \begin{equation*}
        g(t)\approx \frac{t^2}{2},  \quad t\text{ is near to } 0.        
    \end{equation*}    
    Considering the integral expression $I_0(t)=\frac{1}{\pi} \int_{-1}^{1} \frac{e^{-t\theta}}{\sqrt{1-\theta^2}}d\theta$ (see \cite[p.237]{nico_temme_bessel}), we obtain
      \begin{equation}\label{i_o-I_1}
        I_0(t)-I_0'(t)=\frac{1}{\pi} \int_{-1}^{1} \frac{(1+\theta)e^{-t\theta}}{\sqrt{1-\theta^2}}d\theta\rightarrow 0\text{ as }t\rightarrow\infty ,
    \end{equation}
    by using the dominated convergence theorem. 
    Therefore, as $t\rightarrow\infty$, we have:
    \begin{equation*}
        I_0(t)\approx I_0'(t)\text{ and hence }g(t)\approx t.
    \end{equation*}
    

    \noindent Thus, we have the following approximation:
    \begin{equation*}
        g(t)\approx \begin{cases}
            \frac{t^2}{2} &  t\text{ is near to } 0,\\
            t & t \text{ is near }\infty.
        \end{cases}
    \end{equation*}
    Consequently, 
    \begin{equation*}
        \mu_{0}(B_R)\approx\begin{cases}
             \frac{1}{\log\, R} & \text{ when  $R$ is near to } 1,\\
             \sqrt{\frac{2}{\log\, R}} & \text{ when $R$ is near $\infty$}.
        \end{cases}
    \end{equation*}
    Therefore, from Theorem \ref{radial_small_version}, we obtain the following approximation for $\Tilde{\tau}_1(B_R)$:
    \begin{equation*}
        \Tilde{\tau}_1(B_R)\approx\begin{cases} 
                    -R^2(\log\, R)^2 & \text{ when  $R$ is near } 1,\\
                 -\frac{R^2 \,\log\, R}{2} & \text{ when $R$ is near $\infty$} .
        \end{cases}
    \end{equation*}\qed
\section{Further properties of eigenvalues}
This section provides further discussions on $\tau_1$ and $\Tilde{\tau}_1$. We analyze the monotonicity of $\tau_1$ under some specific domain variations. Additionally, we provide an example where the Faber-Krahn inequality fails for $\Tilde{\tau}_1$. We also study the strict domain monotonicity of $\tau_1$ and $\Tilde{\tau}_1$.

\subsection{Monotonicity of eigenvalues under domain perturbations}\label{translation_rotation}
Similar to Theorem \ref{corlog}, we study the strict monotonicity of $\tau_1$ by using Theorem \ref{fklog}. The key idea is to express a domain variation in terms of an appropriate set of polarizations so that $P_H(\Om)\ncong \Om$ nor $P_H(\Om)\ncong \sigma_H(\Om)$. Motivated by Theorem 1.5 and Theorem 1.8 of \cite{Anoop-Ashok2023}, we provide examples of such perturbations of punctured open sets of the form $\Om\setminus\mathcal{O}\subset\R^2$.

    
    \begin{example}\textbf{Monotonicity of $\tau_1$ in  punctured domains of the form $\Omega\setminus\mathcal{O}$ under the translation of the obstacle $\mathcal{O}$:}\label{translation_statement}
    Let $h\in \mathbb{S}^{1}$ and let $ H_s=\{x\in\mathbb{R}^2:x\cdot h<s\}$. Let $\Omega,\mathcal{O}$ satisfies the following conditions:
    \begin{enumerate}
        \item $\mathcal{O}\subset \Om$  is a closed set with non zero measure.
        \item $P_{H_0}(\Omega)=\Omega$ and $\mathcal{O}$ is Steiner symmetric w.r.t $\partial H_0$.
        \item $diam(\Om)\leq 1$.
    \end{enumerate}
    The translations of $\mathcal{O}$ in the direction of $h$ are denoted by 
    \begin{equation*}
        \mathcal{O}_s=sh+\mathcal{O},\quad s\in\R.
    \end{equation*}
     \begin{center}
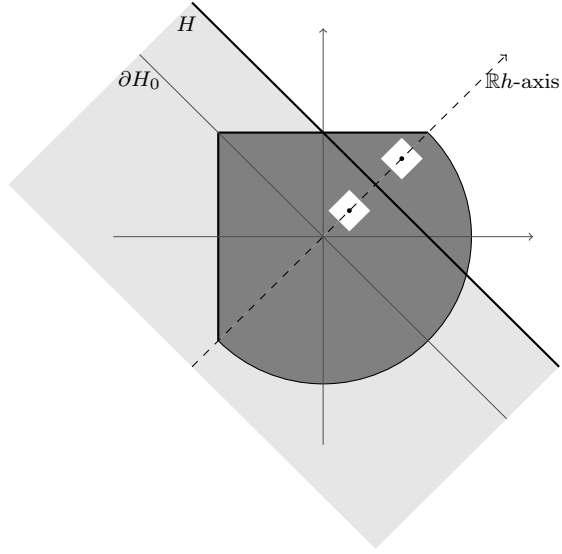

    \captionsetup{type=figure}
        \begin{tikzpicture}[scale=0.69]
            \fill[fill=mygray,shift={(1 cm,1 cm)}] (-3.5,3.5) -- (3.5,-3.5) -- (0,-7) -- (-7,0) -- cycle ;
            \fill[fill=gray] (-2,2) -- (2,2) -- (-2,-2) -- cycle ;
            \draw[-, thick] (-2,2) -- (2,2);
            \draw[-, thick] (-2,2) -- (-2,-2);
            \begin{scope}
                \clip (-3.5,-3.5) -- (0,-7) -- (7,0) -- (3.5,3.5)  -- cycle ;
                \fill[fill=gray] (0,0) circle(2.828427);
                \draw (0,0) circle(2.828427);
             \end{scope}  
             \fill[fill=white] (0.1,0.5) -- (0.5,0.1) -- (0.9,0.5) -- (0.5,0.9) -- cycle ;
             \fill[fill=white] (1.1,1.5) -- (1.5,1.1) -- (1.9,1.5) -- (1.5,1.9) -- cycle ;
              \draw[lblack,->, very thin] (-4,0) -- (4,0);
              \draw[lblack,->, very thin] (0,-4) -- (0,4);
            \draw [->,dashed] (-2.5,-2.5) -- (3.5,3.5);
            \draw [lblack,-,very thin] (-3.5,3.5) -- (3.5,-3.5);
            \draw [-,thick,shift={(1 cm,1 cm)}] (-3.5,3.5) -- (3.5,-3.5);
            \filldraw[black] (0.5,0.5) circle (1pt);
            \filldraw[black] (1.5,1.5) circle (1pt);            
            \begin{scriptsize}    
                \draw (3.8,3) node {$\R h$-axis};
                \draw (-3.5,3) node {$\partial H_0$};
                \draw (-2.6,4.1) node {$H$};
            \end{scriptsize}                
        \end{tikzpicture}
        \captionof{figure}{Translations of the obstacle $\mathcal{O}$ in $h$-direction}
         \label{pic_translation}
    \end{center}    
    We consider domains of the form $\Om_s=\Om\setminus\mathcal{O}_s$. Define $L_\mathcal{O}=\{s\geq 0:P_{H_s}(\Omega)=\Omega,\mathcal{O}_s\subset \Omega\}$. For $s<t$ in $L_{\mathcal{O}}$, we have $P_{H_a}(\Omega_s)=\Omega_t$, where $a=\frac{s+t}{2}$.  One can observe that, $\Om_t\ncong \Om_s$ nor $\Om_t\ncong \sigma_{H_a}(\Om_s)$. By Theorem \ref{fklog}, the largest eigenvalue increases monotonically in $L_{\mathcal{O}}$, 
    \begin{equation*}
        \tau_1(\Omega_s)<\tau_1(P_{H_a}(\Omega_s))=\tau_1(\Omega_t).
    \end{equation*}

\end{example}

\begin{example}\textbf{Monotonicity of $\tau_1$ in  punctured domains  under the rotation of the obstacle $\mathcal{O}$:}\label{rotation_statement}
 Assume that $\Omega$ and $\mathcal{O}$ satisfy the following conditions:
\begin{enumerate}
    \item $\mathcal{O}$ is a closed set with nonzero measure.
    \item $\Omega$ and $\mathcal{O}$ are foliated Schwarz symmetric with respect to the $x$-axis.
    \item $\Omega$ is not radial.
    \item $\text{diam}(\Omega) \leq 1$.
\end{enumerate}
\noindent Let
    \begin{equation*}
        \mathcal{O}_\theta =\left\{\begin{bmatrix}
        \cos\theta & -\sin\theta \\
        \sin\theta & \cos\theta
        \end{bmatrix}\begin{bmatrix}
        x_1  \\
        x_2
        \end{bmatrix} : (x_1,x_2)\in\mathcal{O}  \right\},
    \end{equation*}
    and define
    $\Om_\theta=\Om\setminus\mathcal{O_\theta}$.
    
   \begin{center}
  \captionsetup{type=figure}
      \begin{tikzpicture}
      \filldraw[color=white, fill=medgray, rotate=60](-3.5,-3.5) rectangle (3.5,0);
     \filldraw[color=white, fill=gray,  thin](0,0) circle (3);
     \def\circ1{(0,0) circle (3)};
     \begin{scope}
         \clip\circ1;
         \filldraw[ fill=white,  thick](-3,0) circle (3);
     \end{scope}
     \def\circ2{(-3,0) circle (3)};
     \begin{scope}
         \clip\circ2;
         \filldraw[color=white, fill=medgray, rotate=60](-3.5,-3.5) rectangle (3.5,0);
     \end{scope}
     \fill[fill=white] (1.6,0) -- (2,-0.4) -- (2.4,0) -- (2,0.4) -- cycle ;
     \fill[fill=white,rotate=40] (1.6,0) -- (2,-0.4) -- (2.4,0) -- (2,0.4) -- cycle ;
     \fill[fill=white,rotate=80] (1.6,0) -- (2,-0.4) -- (2.4,0) -- (2,0.4) -- cycle ;
      \draw[lblack,->, very thin] (0,0) -- (3.5,0);
      \draw[lblack,->, very thin,rotate=40] (0,0) -- (3.5,0);
       \draw[lblack,->, very thin,rotate=80] (0,0) -- (3.5,0);
       \draw[->, very thick,rotate=60] (-3.5,0) -- (3.5,0);
     \draw (-0.20,0) node {$0$};
     \draw (4.25,0) node {$x$-axis};
     \draw[black] (3,0) arc (0:120:3);
     \draw[black] (3,0) arc (0:-120:3);
     \draw[black] (0,0) arc (0:-60:3);
     \draw [->,dashed] (2,0) arc (0:100:2);
    \coordinate (A) at (2,0);
    \coordinate (B) at (0,0);
    \coordinate (C) at (1,1);
    \coordinate (D) at (0.174,0.985);
    \pic [draw, ->,angle radius=5mm, "$\theta_1$", angle eccentricity=1.5] {angle = A--B--C};
   \pic [draw, ->,angle radius=9mm, "$\theta_2$", angle eccentricity=1.25] {angle = A--B--D};
  \draw (1.95,3.23) node {$\frac{\theta_1+\theta_2}{2}$};
     \draw (3.25,1.75) node {$H$};
       
      \end{tikzpicture}
      
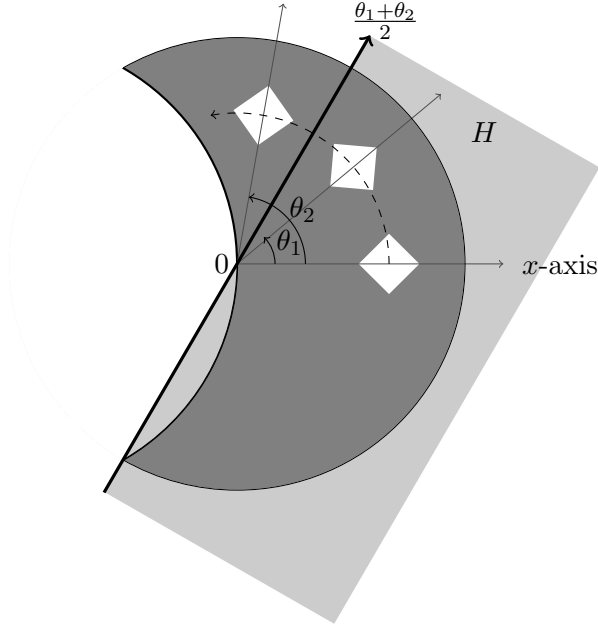
\captionof{figure}{Rotations with respect to $a$}
      \label{pic_rotation}
  \end{center}   
   \noindent Let
   \begin{equation*}
        I_\mathcal{O}=\{\theta\in[0,\pi]: \mathcal{O}_\theta\subset \Omega\}.
    \end{equation*} 
    For $\theta_1,\theta_2\in I_{\mathcal{O}}$ with $\theta_1<\theta_2$, let $H_\theta$ be the polarizer obtained by rotating the lower half plane by an angle $\frac{\theta_1 + \theta_2}{2}$ in the anti-clockwise direction.
\noindent Then, one can verify that  $P_{H_\theta}(\Om_{\theta_1})=\Om_{\theta_2}$. Clearly,  $\Om_{\theta_2}\ncong\Om_{\theta_1}$ and  $\Om_{\theta_2}\ncong \sigma_{H_\theta}(\Om_{\theta_1})$. Therefore, by Theorem \ref{fklog}, we have:
     \begin{equation}
            \tau_1(\Omega_{\theta_1})< \tau_1(P_{H_\theta}(\Omega_{\theta_1}))= \tau_1(\Omega_{\theta_2}).
        \end{equation}  
\end{example}
In the above examples, observe that the transfinite diameter of the domains was fixed. Next, we study the monotonicity of $\Tilde{\tau}_1$ with respect to the diameter.

\begin{example}\label{as_distance_increases_disconnected}
    Let $B_1(\mathbf{0})$ and $B_1(d\mathbf{\mathbf{e_1}})$ be two disjoint balls and let $\Om_d= B_1(\mathbf{0})\cup  B_1(d\mathbf{e_1})$, where $d$ is sufficiently large. 
     \begin{center}
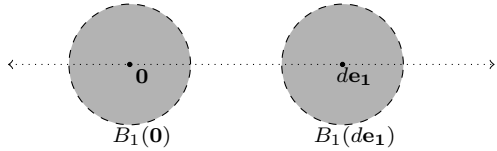

     \captionsetup{type=figure}
     \begin{tikzpicture}[scale=0.8]
            \fill [fill=dgray](0,0) circle (1cm); 
            \draw[color=black,dashed] (0,0) circle [radius=1];
            \fill [fill=dgray](3.5,0) circle (1cm); 
            \draw[color=black,dashed] (3.5,0) circle [radius=1];
             \draw[<->, dotted] (-2,0) -- (6,0);
             \filldraw[black] (0,0) circle (1pt);
             \filldraw[black] (3.5,0) circle (1pt);
               \begin{scriptsize}    
                    \draw (0.2,-0.2) node {$\mathbf{0}$};
                    \draw (3.7,-0.2) node {$ d\mathbf{e_1}$};
                    \draw (0.2,-1.2) node {$ B_1(\mathbf{0})$};
                    \draw (3.7,-1.2) node {$B_1(d\mathbf{e_1})$}; 
            \end{scriptsize}
     \end{tikzpicture}
     \captionof{figure}{$\Om_{d}= B_1(\mathbf{0})\cup B_1(d\mathbf{e_1})$}
 \end{center}

\noindent Define the map $T:\Om_{d_2}\rightarrow\Om_{d_1}$ by
$$T(x)=\begin{cases}
    x,&x\in  B_1(\mathbf{0}),\\
    x-(d_2-d_1)\mathbf{e_1},&x\in B_1(d_2\,\mathbf{e_1}).
\end{cases}$$
Let $u$ be the normalized positive eigenfunction corresponding to $\Tilde{\tau}_1(\Om_{d_1})$. Define $\Tilde{u}$ on $\Om_{d_2}$,
$$\Tilde{u}(x)=u\circ T(x),\quad x\in \Om_{d_2}.$$ Notice that $\|\Tilde{u}\|_{\Om_{d_2}}=\|u\|_{\Om_{d_1}}=1$. Moreover, for any $x,y\in\Om_{d_2}$, we have  $|x-y|\geq |T(x)-T(y)|$. In particular, for $x\in B_1(\mathbf{0})$ and $y\in B_1(d_2\,\mathbf{e_1})$, it holds that $|x-y|> |T(x)-T(y)|$, this leads to:
\begin{equation}\label{dist_eqn1}
    \frac{1}{2\pi}\iint\limits_{\Om_{d_2}\;\Om_{d_2}} \log \frac{1}{|x-y|}\Tilde{u}(x)\Tilde{u}(y)dxdy < \frac{1}{2\pi}\iint\limits_{\Om_{d_2}\;\Om_{d_2}} \log \frac{1}{|T(x)-T(y)|}\Tilde{u}(x)\Tilde{u}(y)dxdy.
\end{equation}
By applying the change of variables, we obtain:
\begin{align}\label{dist_eqn2}
    \frac{1}{2\pi}\iint\limits_{\Om_{d_2}\;\Om_{d_2}} \log \frac{1}{|T(x)-T(y)|}\Tilde{u}(x)\Tilde{u}(y)dxdy &=\frac{1}{2\pi}  \iint\limits_{\Om_{d_1}\;\Om_{d_1}} \log \frac{1}{|x-y|}u(x)u(y)dxdy\nonumber\\
   &= \Tilde{\tau}_1(\Om_{d_1})    .  
\end{align}
 From the inequalities \eqref{dist_eqn1} and \eqref{dist_eqn2}, we get
 \begin{equation*}
     \Tilde{\tau}_1(\Om_{d_2})<\Tilde{\tau}_1(\Om_{d_1}).
 \end{equation*}
Next, consider the constant function $c=\frac{1}{\sqrt{2\pi}}$ on $\Om_d$, in \eqref{pevneg}
\begin{align}\label{dist_to_infty}
     \Tilde{\tau}_1(\Om_{d})    
     &\leq \frac{1}{4\pi^2}\iint\limits_{ B_1(\mathbf{0})\; B_1(\mathbf{0})} \log \frac{1}{|x-y|}dxdy +\frac{1}{2}\log \frac{1}{d-2}
     +\frac{1}{4\pi^2}\iint\limits_{  B_1(d\mathbf{e_1})\;  B_1(d\mathbf{e_1})} \log \frac{1}{|x-y|}dxdy\nonumber\\
 \end{align}
 As $d\rightarrow\infty$, the second term in \eqref{dist_to_infty} tends to $-\infty$, while the first and third terms remain the same for any $d$. Thus, we conclude that  $\Tilde{\tau}_1(\Om_{d})\rightarrow-\infty$. 

 \end{example}
In the previous example, we discussed disconnected open sets. The following example shows that, even in the class of domains with transfinite diameter bigger than one and fixed measure, $\Tilde{\tau}_1(\Om)$ is not bounded below. This will imply that Faber-Krahn inequality fails for $\Tilde{\tau}_1$.


\begin{example}\textbf{(Faber-Krahn inequality fails for $\Tilde{\tau}_1$)}\label{as_distance_increases_connected}
      For $c>0$, we construct a class of domains and show that
 \begin{equation*}
\inf\{\Tilde{\tau}_1(\Om):|\Om|=c \text{ and } T_{diam}(\Om)>1\}=-\infty.
\end{equation*}
      Let $S_1,S_d$ be the squares centred at the origin and at $(d,0)$ with the measures $\frac{c}{3}$. Let $P_d$ 
     be the rectangle with  measure $\frac{c}{3}$ and connects $S_1$ and $S_d$. For sufficiently large $d$, define $\Om_d=S_1\cup P_{d}\cup S_d$ as shown in Fig. \ref{pic_square}.
     
 \begin{center}
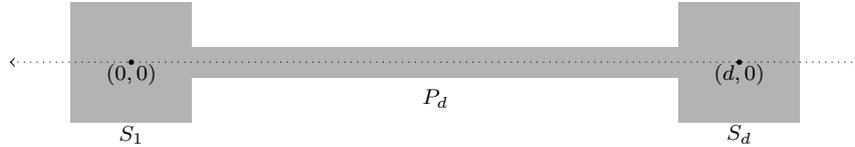

     \captionsetup{type=figure}
     \begin{tikzpicture}[scale=0.8]
             \fill [fill=dgray](-1,-1) rectangle (1,1); 
             \fill [fill=dgray](1,-0.25) rectangle (9,0.25); 
             \fill [fill=dgray](9,-1) rectangle (11,1); 
            
             \draw[<->, dotted] (-2,0) -- (12,0);
             \filldraw[black] (0,0) circle (1pt);
             \filldraw[black] (10,0) circle (1pt);
               \begin{scriptsize}    
                    \draw (0,-1.2) node {$S_1$};
                    \draw (5,-0.6) node {$P_{d}$};    
                    \draw (10,-1.2) node {$S_d$}; 
                    \draw (0,-0.2) node {$(0,0)$};
                    \draw (10,-0.2) node {$(d,0)$};
            \end{scriptsize}
     \end{tikzpicture}
     \captionof{figure}{$\Om_{d}=S_1\cup P_{d}\cup S_d$}
     \label{pic_square}
 \end{center}
\noindent Consider $u_d(x) =\sqrt{\frac{3}{2c}}\,\,\chi_{S_1\cup S_d}$. Now, \eqref{pevneg} yields 
\begin{align*}
  \Tilde{\tau}_1(\Om_d) 
     &\leq \frac{3}{4\pi c}\iint\limits_{S_1\; S_1} \log \frac{1}{|x-y|}dxdy + \frac{c}{6\pi}\log \frac{1}{d-2}
     +\frac{3}{4\pi c}\iint\limits_{S_d\; S_d} \log \frac{1}{|x-y|}dxdy  
\end{align*}
Integrals in the first and third terms remain the same for any $d$, while the second term tends to $-\infty$ as $d\rightarrow \infty$. Therefore, 
     \begin{equation}
         \lim_{d\rightarrow\infty} \Tilde{\tau}_1(\Om_d) =-\infty.
     \end{equation}
\end{example}

\subsection{The Domain monotonicity}

\par   The strict domain monotonicity of $\Tilde{\tau}_1$ was already established in \cite[Theorem 1]{troutman_planar}. Here, we provide an alternative proof of this result. In \cite{estimation_proceedings}, the domain monotonicity of $\tau_1$ is established. We prove the strict domain monotonicity of $\tau_1$. Before formally stating these results, we first prove an important property of the eigenpairs $(\tau_1, w)$ and $(\Tilde{\tau}_1, \Tilde{w})$.

\begin{proposition}\label{rayleighquotient_pos_iff}
    Let $w,\Tilde{w}\in L^2(\Omega)$ such that $\|w\|_{L^2(\Om)}=\|\Tilde{w}\|_{L^2(\Om)}=1$ and let $$E(w)=\tau_1(\Om)\text{ and } E(\Tilde{w})=\Tilde{\tau}_1(\Om).$$ Then, $w$ is an eigenfunction corresponding to $\tau_1(\Om)$ and $\Tilde{w}$ is an eigenfunction corresponding to $\Tilde{\tau}_1(\Om)$.
\end{proposition}
\begin{proof}
      Let $v\in L^2(\Om)$ and $t>0$. 
  Then, $w+tv\in L^2(\Om)$ and by $\eqref{pevpos}$ we get
  \begin{equation*}
      \frac{E (w+tv) }{\int_\Om (w+tv)^2(x)dx}\leq \tau_1(\Om).
  \end{equation*}
  Noting that $w$ is an eigenfunction corresponding to $\tau_1(\Om)$, by expanding $E (w+tv)$, we get:
  \begin{equation*}
       \frac{t}{\pi}\iint\limits_{\Om\;\Om}\log\frac{1}{|x-y|} w(x)v(y)dxdy + \frac{t^2}{2\pi}\iint\limits_{\Om\;\Om}\log\frac{1}{|x-y|} v(x)v(y)dxdy
  \end{equation*}
  \begin{equation}
      \leq \tau_1(\Om)\left[ 2t\int_\Om w(x)v(x)dx+t^2 \int_\Om v(x)^2dx\right].
  \end{equation}
Now,  divide by $2t$ and let $t\rightarrow 0$ to obtain:  
  \begin{equation}
      \frac{1}{2\pi}\iint\limits_{\Om\;\Om}\log\frac{1}{|x-y|} w(x)v(y)dxdy\leq\tau_1(\Om) \int_\Om w(x)v(x)dx\,\; \forall\, v\in L^2(\Omega).
  \end{equation}
  Substituting $-v$ instead of $v$ yields the other way inequality. Thus by $\eqref{var_char}$, $w$ is an eigenfunction corresponding to $\tau_1(\Om)$. Similarly, it can be shown that $\Tilde{w}$ is an eigenfunction corresponding to $\Tilde{\tau}_1(\Om)$.
\end{proof}

\noindent Next, we prove the strict domain monotonicity for $\tau_1$.
\begin{theorem}
    Let $ \Om_1\subseteq \Om_2$. Then,
    \begin{equation}
        \tau_1( \Om_1)\leq \tau_1(\Om_2).
    \end{equation}
    Furthermore, if $diam(\Om_2)\leq 1$ and $|\Om_2\setminus  \Om_1|\neq 0$, then 
    \begin{equation}
        \tau_1( \Om_1)<\tau_1(\Om_2).
    \end{equation}    
\end{theorem}
\begin{proof}
Let $u$ be an eigenfunction corresponding to $\tau_1( \Om_1)$ and $\Tilde{u}$ be its zero extension to $\Om_2$. Then,
   \begin{align*}
       \tau_1( \Om_1)&=\frac{1}{2\pi}\iint\limits_{ \Om_1\; \Om_1}\log\frac{1}{|x-y|} u(x)u(y)dxdy=\frac{1}{2\pi}\iint\limits_{\Om_2\;\Om_2}\log\frac{1}{|x-y|} \Tilde{u}(x)\Tilde{u}(y)dxdy\leq \tau_1(\Om_2)
   \end{align*}
   Moreover,  as $|\Om_2\setminus  \Om_1|\neq 0$ and $diam(\Om_2)\leq 1$, by Proposition \ref{prop-non-neg}, $\Tilde{u}$ cannot be an eigenfunction corresponding to $\tau_1(\Om_2)$. Therefore, by Proposition \ref{rayleighquotient_pos_iff},
   \begin{align*}
       \tau_1( \Om_1)&=\frac{1}{2\pi}\iint\limits_{\Om_2\;\Om_2}\log\frac{1}{|x-y|} \Tilde{u}(x)\Tilde{u}(y)dxdy<\tau_1(\Om_2).
   \end{align*}
\end{proof}

\noindent Next, we prove strict domain monotonicity for $\Tilde{\tau}_1(\Om)$.
\begin{theorem}\label{domain_monotonicity_tilde}
    Let $ \Om_1\subseteq \Om_2$ and let $T_{diam}( \Om_1)>1$. Then
    \begin{equation}
        \Tilde{\tau}_1( \Om_1)\geq \Tilde{\tau}_1(\Om_2).
    \end{equation}
    Moreover, if $|\Om_2\setminus  \Om_1|\neq 0$, then we have:
    \begin{equation}
        \Tilde{\tau}_1( \Om_1)>\Tilde{\tau}_1(\Om_2).
    \end{equation}
\end{theorem}
\begin{proof}
     Let $v$ be an eigenfunction corresponding to $\Tilde{\tau}_1( \Om_1)$ and $\Tilde{v}$ be its zero extension to $\Om_2$ and proceed as in the above theorem. Then,
   \begin{align*}
       \Tilde{\tau}_1( \Om_1)
       &\geq\Tilde{\tau}_1(\Om_2).
   \end{align*}
   Since $|\Om_2\setminus  \Om_1|\neq 0$, by \cite[Theorem 1]{troutman1967} $\Tilde{v}$ cannot be an eigenfunction corresponding to  $\Tilde{\tau}_1(\Om_2)$.  Therefore, by Proposition \ref{rayleighquotient_pos_iff},
   \begin{align*}
       \Tilde{\tau}_1( \Om_1)&=\frac{1}{2\pi}\iint\limits_{\Om_2\;\Om_2}\log\frac{1}{|x-y|} \Tilde{v}(x)\Tilde{v}(y)dxdy>\Tilde{\tau}_1(\Om_2).
   \end{align*}
\end{proof}

\begin{remark}
In \cite[Theorem 1.1]{estimation_proceedings}, the authors approximated the positive eigenvalues of domains by considering the eigenvalues of discs, assuming that $\Lom$ is a positive operator. Following a similar approach, we approximate $\Tilde{\tau}_1$ using the negative eigenvalue of discs. Let $\Om$ be a bounded open set and $x,y\in\Om$  such that  $B_{R_1}(x)\subseteq\Om\subseteq B_{R_2}(y)$, where $R_1>1$. 

\begin{center}
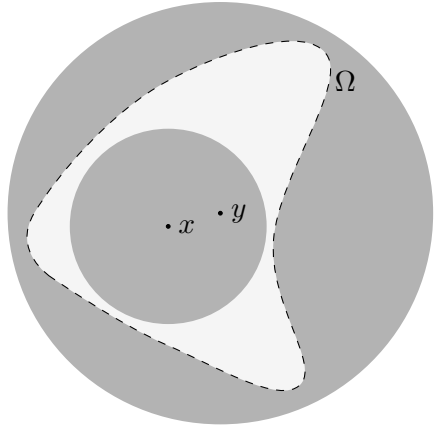

  \captionsetup{type=figure}
  \begin{tikzpicture}[scale=0.7]
        \def\circ1{(1,0) circle (4)};
         \fill[fill=dgray] (1,0) circle (4);
         \coordinate  (O) at (0,0);
        \fill[fill=vlgray]  plot[smooth, tension=0.7] coordinates {(-2.2,-1.2) (-2.5,0) (0,2.5)   (3,3) (2,-0.5) (2.5,-3.25) (0,-2.5)  (-2.2,-1.2) };
        \draw [-,dashed]  plot[smooth, tension=0.7] coordinates {(-2.2,-1.2) (-2.5,0) (0,2.5)   (3,3) (2,-0.5) (2.5,-3.25) (0,-2.5)  (-2.2,-1.2) };
          \fill[fill=dgray] (0.02,-0.25) circle (1.85); 
          \draw [fill] (1,0) circle [radius=1pt];
           \draw (1.35,0) node {$y$};
           \draw [fill] (0.02,-0.25) circle [radius=1pt];
            \draw (0.37,-0.25) node {$x$};
            \draw (3.35,2.5) node {$\Om$};
     \end{tikzpicture}
      \captionof{figure}{$B_{R_1}(x)\subseteq \Om \subseteq B_{R_2}(y)$}
  \end{center}
Then, by Theorem \ref{domain_monotonicity_tilde} and  Theorem \ref{radial_small_version}, we have:
 \begin{equation*}
        -\frac{R_2^2}{(\mu_{0}(B_{R_2}))^2}\leq \Tilde{\tau}_1(\Om)\leq-\frac{R_1^2}{(\mu_{0}(B_{R_1}))^2}.
    \end{equation*}

\noindent In particular, for a domain $\Om$, such that $B_{R-\epsilon}\subseteq \Om\subseteq B_{R+\epsilon} $, for $R>1$ and sufficiently small $\epsilon>0$, $\Tilde{\tau}_1(\Om)$ can be approximated as follows:
        \begin{equation*}
        -\frac{(R+\epsilon)^2}{(\mu_{0}(B_{R+\epsilon}))^2}\leq \Tilde{\tau}_1(\Om)\leq-\frac{(R-\epsilon)^2}{(\mu_{0}(B_{R-\epsilon}))^2}.
    \end{equation*}

\end{remark}

    


\section{Other remarks and open problems}

In this section, we present isoperimetric inequalities for the Riesz potential operator. One can view the Riesz potential operator as a generalization of the Newtonian potential operator. Furthermore, we state some conjectures concerning transfinite diameter and  $\Tilde{\tau}_1$.
\subsection{Riesz potential operator}

 In this subsection, we briefly introduce Riesz potential operator and state the related results. In \cite[Lemma 3.2]{suragan2016riesz}, a Rayleigh–Faber–Krahn theorem for the Riesz potential operator was proven. Recently,  a Faber-Krahn inequality for the Riesz potential operator is established among the classes of triangular domains and quadrilateral domains in \cite[Theorem 1.1]{rajesholiver}. We extend the results for the Logarithmic potential operator to the Riesz potential operator. The proof techniques used for the Logarithmic potential operator also apply to the Riesz potential operator.

\begin{definition}{\textbf{(Riesz potential operator).}}
    Let $\Omega \subseteq \mathbb{R}^N$ be a bounded open set and $N\geq 2$. Let $0<\alpha<N$.
 The Riesz potential operator on $L^2(\Omega)$ is defined by
\begin{align}
    \mathcal{R}u(x)&:=\int_\Omega \frac{u(y)}{|x-y|^{N-\alpha}}dy. 
\end{align}
\end{definition}



$\mathcal{R}$ is a positive, self-adjoint, and compact operator on $L^2(\Omega)$. 
\noindent The eigenpair $(\la,u)$ satisfy the following equation:
\begin{equation}
     \iint\limits_{\Om\;\Om} \frac{u(x)\phi(y)}{|x-y|^{N-\alpha}}dxdy=\la \int_\Omega u(x)\phi(x)dx, \,\; \forall\, \phi\in L^2(\Omega).
\end{equation}
Moreover, the largest eigenvalue,  $\la_1(\Omega)$ has the following variational characterization: 
\begin{equation}\label{pevr}
    \lambda_1(\Om)=\sup \left\{ \iint\limits_{\Om\;\Om} \frac{u(x)u(y)}{|x-y|^{N-\alpha}}dxdy:u\in L^2(\Omega),\|u\|_{L^2(\Omega)}=1\right\}.
    \end{equation}
The corresponding eigenfunction is not sign-changing \cite[ Lemma 3.1]{suraganfirstsecond}. The following reverse Faber-Krahn inequalities under polarization can be proved for $\la_1$ of $\mathcal{R}$. The proof follows from the positivity of the eigenfunction \cite[ Lemma 3.1]{suraganfirstsecond} and Proposition \ref{riezpol}.

\begin{theorem}\label{fkriesz}
        Let $\Omega \subseteq \mathbb{R}^N$ be a bounded open set, $H$ be a polarizer, and $P_H(\Omega)$ be the corresponding polarized domain. Let $N\geq 2$ and $0<\alpha<N$. Then, 
        \begin{equation}\label{poline}
            \lambda_1(\Omega)\leq \lambda_1(P_H(\Omega)).
        \end{equation}
        Furthermore, the equality holds in $\eqref{poline}$ only if $P_H(\Omega)\cong\Omega$(almost equal) or $P_H(\Omega)\cong\sigma_H(\Omega)$(almost equal).
    \end{theorem}
   Consequently, similar to \eqref{fkaliterlog}, one can obtain the following corollary for  $\la_1$ by approximating Schwarz symmetrization via polarization.
\begin{corollary}
    Let $\Omega\subset\R^N$ be a bounded domain. Let $N\geq 2$ and $0<\alpha<N$. Then, 
        \begin{equation}\label{fkaliterriesz}
            \la_1(\Omega)\leq \la_1(\Omega^*).
        \end{equation}
        Equality holds in \eqref{fkaliterriesz}  only if $\Omega\cong\Omega^*$ up to a translation.
\end{corollary}


\begin{remark}
Like the Logarithmic potential operator, the Riesz potential operator exhibits monotonicity of the largest eigenvalue under various domain perturbations.
    \begin{enumerate}[(i)]
        \item \textbf{The monotonicity of $\la_1$ for eccentric annular region:} The case of an eccentric annular region for $\Lom$ is discussed in Theorem \ref{corlog}. A similar result also holds for $\mathcal{R}$. Let $0<r<R$ and $e_1=(1,0,0,\cdots)\in\mathbb{S}^{N-1}$. Recall $\Omega_t=B_R(0)\setminus \overline{B_r(te_1)}$. Then, $\la_1(\Omega_t)$ is strictly increasing for $t\in [0,R-r)$. In this case, the restriction on the outer radius is not required since a positive eigenfunction exists.

        \item \textbf{The monotonicity of $\la_1$ for punctured domains, as the obstacle is translated:} Let $\Omega_s=\Omega\setminus\mathcal{O}_s$ and $L_\mathcal{O}$ be as mentioned in Remark \ref{translation_statement}. For  $s<t$ in $L_{\mathcal{O}}$, we have $P_{H_a}(\Omega_s)=\Omega_t$, where $a=\frac{s+t}{2}$. Then, by Theorem \ref{fkriesz}, $\la_1(\Om_s)$ increases monotonically as follows:
             \begin{equation*}
                \la_1(\Omega_s)<\la_1(\Omega_t).
             \end{equation*}

        \item \textbf{The monotonicity of the eigenvalues for punctured domains, as the obstacle is rotated:}. Let $\Om_\theta=\Om\setminus\mathcal{O_\theta}$ and $I_{\mathcal{O}}$ be as defined in Remark \ref{rotation_statement}. For $\theta_1<\theta_2$ in $I_{\mathcal{O}}$, recall $P_{H_\theta}$ from Remark \ref{rotation_statement} such that $P_{H_\theta}(\Om_{\theta_1})=\Om_{\theta_2}$. Then, by Theorem \ref{fkriesz},
        \begin{equation}
            \la_1(\Omega_{\theta_1})< \la_1(\Omega_{\theta_2}).
        \end{equation}  
    \end{enumerate}
\end{remark}



\subsection{Some open problems}
As previously mentioned, the transfinite diameter of a domain determines whether the operator is positive or not. In \cite[Theorem 3]{troutman1967}, author proved that $\Lom$ is a positive operator on $L^2(\Om)$ if and only if $T_{diam}(\Om)\leq 1$. Moreover, if $T_{diam}(\Om)>1$, then there exists a unique negative eigenvalue $\Tilde{\tau}_1(\Om)$ for $\Lom$\cite[Theorem 2]{troutman1967}. Next, we provide an example of a compact set $E$ for which $T_{diam}(E)>1$, but $T_{diam}(P_H(E))\leq 1$.
\begin{example}\label{transfinite_diam_example} 
Consider an ellipse $E$ with semi-axes $2$ and $0.25$. The transfinite diameter of $E$ is given by $T_{diam}(E)=\frac{2+0.25}{2}=1.125$. Let $E^*$ be the disc with radius $\frac{1}{\sqrt{2}}$, which has the same area as E. For $E^*$, $T_{diam}(E^*)=\frac{1}{\sqrt{2}}\approx 0.71$. There exists a sequence of polarizers $H_n$ and polarized domains $E_n=P_{H_n H_{n-1}\cdots H_1}(E)$  that converges to $E^*$ with respect to the Hausdorff distance \cite[Lemma 7.2]{brockpol}. i.e; 
    \begin{equation}
        \lim_{n\rightarrow \infty}d(E_n,E^*)=0,
    \end{equation}
    where the Hausdorff distance $d$ between two compact sets $K$ and $M$ is defined as
    \begin{equation*}
        d(K,M):=\inf\{r>0:K\subset M_r,\, M\subset K_r\}, \quad M_r=M+B_r\text{ and } K_r=K+B_r.
    \end{equation*}
    Thus, for large $n$, $E_n \subset B_1$. Consequently, there exists an $N\in\N$ such that $T_{diam}(E_N)>1$ and $$T_{diam}(P_{H_{N+1}}(E_N))=T_{diam}(E_{N+1})\leq 1.$$
\end{example}
It is noted in \cite[Note A, p.188 $\&$ p.192]{polyaszego} and proved in \cite[p.606]{wolontis}, that Schwarz and Steiner symmetrizations decreases the transfinite diameter. We conjecture that polarization also decreases the transfinite diameter. 


\begin{conjecture}
    Let $E$ be a compact subset of $\R^2$ and $H$ be a polarizer. Then, 
    \begin{equation}\label{conjecture_trans}
        T_{diam}(P_H(E))\leq T_{diam}(E).
    \end{equation}
    The equality holds in \eqref{conjecture_trans} only if $P_H(E)\cong E$ or $P_H(E)\cong \sigma_H(E)$. 
\end{conjecture}

\par It is natural to ask whether the reverse Faber-Krahn inequality holds for $\Tilde{\tau}_1$?
If $T_{diam}(\Om)\leq 1$ and $T_{diam}(P_H(\Om))\leq 1$, then $$\Tilde{\tau}_1(\Om)=\Tilde{\tau}_1(P_H(\Om))=0.$$ 
Similarly, if $T_{diam}(\Om)>1$ and $T_{diam}(P_H(\Om))\leq 1$(see Example \ref{transfinite_diam_example}), then $$\Tilde{\tau}_1(\Om)<0\text{ and }\Tilde{\tau}_1(P_H(\Om))=0,$$ so the inequality 
$$\Tilde{\tau}_1(\Omega)\leq \Tilde{\tau}_1(P_H(\Omega)),$$ holds trivially in these cases. Based on the above observations, we anticipate the following reverse Faber-Krahn  inequality for $\Tilde{\tau}_1(\Om)$:
\begin{conjecture}\label{conj_faber}
    Let $\Om$ be a bounded domain and $H$ be a polarizer. Then,
        \begin{equation}\label{pollogtilde}
            \Tilde{\tau}_1(\Omega)\leq \Tilde{\tau}_1(P_H(\Omega)).
        \end{equation}
\end{conjecture}






The challenging part is to construct a function $u$ in $L^2(\Om)$ with $\|u\|_{L^2(\Om)}=1$ from a function $v$ in $L^2(P_H(\Om))$ with $\|v\|_{L^2(P_H(\Om))}=1$, such that the following Riesz-type inequality holds:
 \begin{equation}
        \iint\limits_{\Om\;\Om}\log \frac{1}{|x-y|} u(x)u(y) dxdy \leq \int\limits_{P_H(\Om)} \int\limits_{P_H(\Om)} \log \frac{1}{|x-y|}  v(x)v(y) dxdy.
    \end{equation}

Faber-Krahn-type inequalities can also be studied by fixing the area or the transfinite diameter. We conjecture the following reverse Faber-Krahn inequalities for the Logarithmic operator.
\begin{conjecture}
    Let $\Om$ be a bounded domain. Then,
    \begin{equation}\label{fkaliterlogtilde}
\tilde{\tau}_1(\Omega) \leq \tilde{\tau}_1(\Omega^*).
\end{equation}
Additionally, for a domain $\Om$ with transfinite diameter $R$, 
\begin{equation}\label{fkaliterlogtilde_transfinite}
\tilde{\tau}_1(\Omega) \leq \tilde{\tau}_1(B_R).
\end{equation}
\end{conjecture}
Observe that, if $|\Om|\leq \pi$, then $T_{diam}(\Om^*)\leq 1$. Thus,  $\Tilde{\tau}_1(\Om^*)=0$ and hence \eqref{fkaliterlogtilde} trivially holds.





\appendix

\appendix
\section{Eigenvalues of Logarithmic Potential on a disc}\label{Appendix}

In \cite{anderson}, Anderson and others studied the eigenvalues of Logarithmic potential on a unit disc. They showed that the largest eigenvalue has a multiplicity of three. Motivated by this, we study eigenvalues on a disc with arbitrary radius $R$. For the sake of computational simplicity in this section, we treat $\Omega$ as a subset of $\mathbb{C}$. 
\begin{definition}
    We say an eigenvalue $\tau$ is a radial eigenvalue if there is a radial eigenfunction corresponding to it, while a non-radial eigenvalue is an eigenvalue with a non-radial eigenfunction corresponding to it. An eigenvalue $\tau$ can be both radial and non-radial.
\end{definition}
 First, we describe the radial eigenvalues of $\Lom$ on $B_R$. Before that, we recall the following functions. The Bessel function of order $n$,
\begin{equation}
    J_n(x)=\sum_{m=0}^\infty \frac{(-1)^m}{m!(n+m)!}\left(\frac{x}{2}\right)^{2m+n},
\end{equation} 
and the modified Bessel function of order zero,
\begin{equation}
    I_0(x)=\sum_{m=0}^\infty \frac{1}{(m!)^2}\left(\frac{x}{2}\right)^{2m}.
\end{equation}
Next, we state a more detailed version of Theorem \ref{radial_small_version}, and we will prove Theorem \ref{radial} in this section. 
\begin{theorem}\label{radial}
    Let $B_R$ be the disc with radius $R$ and centre at the origin. Then, the set of all positive radial eigenvalues $\tau_{0,n}$ of $\Lom$  and the corresponding radial eigenfunctions(up to a constant multiple) $u_{0,n}$ on $B_R$ are given by
    \begin{equation*}
        \tau_{0,n}(B_R)=\frac{R^2}{(\mu_{0,n}(B_R))^2},\quad u_{0,n}(x)=J_0\left(\frac{\mu_{0,n}(B_R)}{R}|x|\right),\quad n\in\N,
    \end{equation*}
    where $\mu_{0,n}(B_R)$ is the  n$^{th}$ zero of the  function $
        J_0(t)-\log R\, t J_0'(t).
   $
   In addition,  for $R>1$, there exists a unique negative eigenvalue $\Tilde{\tau}_1(B_R)$ and the corresponding eigenfunction(up to a constant multiple) are given by
    \begin{equation*}
        \Tilde{\tau}_1(B_R)=-\frac{R^2}{\mu_{0}(B_R)^2}, \quad \Tilde{u}(x)=I_0\left(\frac{\mu_{0}(B_R)}{R}|x|\right),
    \end{equation*}
    where $\mu_{0}(B_R)$ is the unique zero of the  function
    $I_0(t)-\log R\, t I_0'(t).  $
   \end{theorem}

    The authors in \cite{suraganvolume} investigate the eigenvalues of the Logarithmic potential on discs, where all eigenvalues are positive. This paper expands its study to include a unique negative and positive eigenvalue on discs $B_R$, even when $\Lom$ is not a positive operator on $L^2(B_R)$. We derive a boundary integral using the approach outlined by \cite{anderson}, used for the unit disc. Next, we will prove some lemmas and propositions required to establish Theorem \ref{radial}.

   \begin{lemma}\label{lapla_ef}
     Let $\Omega$ be a bounded domain with $C^2$ boundary. Then $(\tau,u)$ is an eigenpair for $\Lom$ if and only if $u$ satisfies the equation(in the classical sense):
     \begin{equation}\label{lap_u=la_u}
        -\Delta u=\frac{1}{\tau} u \quad \text{ in } \Omega,
    \end{equation} 
    and the boundary condition:
     \begin{equation}\label{ef_bdry}
         \int_{\partial \Omega} u(y) \frac{\partial}{\partial\eta_y} \log \frac{1}{|x-y|} d\sigma_y- \int_{\partial \Omega} \frac{\partial u}{\partial\eta_y}\log \frac{1}{|x-y|} d\sigma_y=0, \forall\, x\in \Omega.
    \end{equation}
   \end{lemma}
\begin{proof}   
Since the boundary of $\Omega$ is $C^2$,  the  Green’s function $g$ of Laplacian on $\Omega$ exists and it is given by 
$$g(x,y)=\frac{1}{2\pi} \log\frac{1}{|x-y|}-h^x(y),$$  
where $h^x$ is such that 
     \begin{align*}
         \Delta h^x &=0\text{ in }\Omega; \qquad 
         h^x(y)= \frac{1}{2\pi} \log\frac{1}{|x-y|}   \text{ on }\partial \Omega  .    
     \end{align*}
    Let $(Gu)(x)=\int_\Omega g(x,y)u(y)dy.$
    \noindent  Notice that,
    \begin{equation}\label{Gf_Lf}
        (Gu)(x)=\Lom u(x)- \int_{\Omega} h^x(y)u(y)dy.
    \end{equation}
    Then by using Green’s representation formula\cite[p.19]{gilbarg_trudinger}, we have:
   \begin{align}
                 G(\De u)(x)&= \int_\Omega g(x,y)\De u(y)dy=
          - u(x)-  \int_{\pa\Om} u(y) \frac{\pa g(x,y) }{\pa \eta_y } d\si_y \label{Gu=nu_u+} \nonumber\\
          &= -u(x)- \frac{1}{2\pi} \int_{\pa\Om} u(y) \frac{\partial}{\partial\eta_y} \log \frac{1}{|x-y|}d\sigma_y+ \int_{\partial\Omega} u(y)\frac{\partial h^x(y)}{\partial\eta_y}d\sigma_y. 
            \end{align}
            Furthermore, using Green’s second identity, see, for example, \cite[p.18]{gilbarg_trudinger},
             \begin{equation}\label{h}
            \int_\Omega h^x(y)\Delta u(y)dy= \frac{1}{2\pi} \int_{\partial\Omega} \log \frac{1}{|x-y|}\frac{\partial u(y)}{\partial\eta}d\sigma_y -\int_{\partial\Omega} u(y)\frac{\partial h^x(y)}{\partial\eta} d\sigma_y.
         \end{equation}
         Therefore, from \eqref{Gf_Lf}
         \begin{equation}\label{LDelta_u}
             \Lom(\Delta u)(x)= -u(x)- \frac{1}{2\pi} \int_{\pa\Om} u(y) \frac{\partial}{\partial\eta_y} \log \frac{1}{|x-y|}d\sigma_y+\frac{1}{2\pi} \int_{\partial\Omega} \log \frac{1}{|x-y|}\frac{\partial u(y)}{\partial\eta}d\sigma_y.
         \end{equation}
    Let $\tau$ be an eigenvalue of $\Lom$ and $u$  be an eigenfunction corresponding to $\tau$. It can be easily shown that, for any $f\in L^2(\Omega)$, $\Lom f\in C^\alpha (\Omega)$ with $0<\alpha<1$, see for example \cite[p.2]{troutman1967}. Thus $u\in C^\alpha(\Omega)$ and $u$ is bounded. Now, by \cite[Lemma 4.2, Theorem 4.6]{gilbarg_trudinger} we get $ \Lom u\in  C^{2,\alpha}(\Omega)$ and $-\Delta \Lom u=u$ in $\Omega$. Thus $u\in C^2(\overline{\Omega})$ and $-\Delta \tau u= u  \text{ in } \Omega$ pointwise. Therefore, $\tau\neq 0$ and 
    \begin{equation}\label{laplace_ev}
       - \Delta  u= \frac{1}{\tau} u.
    \end{equation}
      Moreover,
       \begin{align}
                Gu(x)&= G(-\tau\De u)(x)=-\tau G(\De u)(x),
            \end{align}
     and
    \begin{align}\label{h^x_y}
      \Lom u(x)&-  \int_{\Omega} h^x(y)u(y)dy
      =\tau  u(x)+ \tau  \int_\Omega h^x(y)\Delta u(y)dy.\nonumber
    \end{align}
 From $\eqref{Gf_Lf}$, $\eqref{Gu=nu_u+}$ and \eqref{h}  we can easily conclude that,
    \begin{equation*}
         \int_{\partial \Omega} u(y) \frac{\partial}{\partial\eta_y} \log \frac{1}{|x-y|} d\sigma_y- \int_{\partial \Omega} \log \frac{1}{|x-y|}\frac{\partial u(y)}{\partial\eta_y} d\sigma_y=0, \forall\, x\in \Omega.
    \end{equation*}
    Conversely, assume that $u$ satisfies $\eqref{lap_u=la_u}$ and $\eqref{ef_bdry}$. Then, 
    \begin{align}
        \Lom u(x)&=-\tau \Lom(\Delta u)
        =\tau u(x),\,\,\forall\, x\in \Om ,
    \end{align}
    follows from \eqref{lap_u=la_u},$\eqref{LDelta_u}$ and \eqref{ef_bdry}.
\end{proof}

   \begin{lemma}\label{ball_lemma1}
   On $B_R$, $\tau$ is an eigenvalue of $\Lom$ if and only if there exists an $m\in \N_0$,  and $\Phi$  on $(0,R)$ satisfying 
     \begin{equation}\label{slproblemradial}
        r^2\Phi''+r\Phi'+\left[\frac{1}{\tau} r^2-m^2\right]\Phi=0,
    \end{equation}
    along with the boundary condition: \\
    \begin{equation}\label{m=0}
        \Phi(R)-R\log R \Phi'(R)=0,\text{ if } m=0,
    \end{equation}
    and 
    \begin{equation}\label{m-1onwards}
        m\Phi(R)+R\Phi'(R)=0, \text{ if } m\in\N.
    \end{equation}
   \end{lemma}

 \begin{proof} Let $\tau$ be an eigenvalue of $\Lom$ on $B_R$. By Lemma \ref{lapla_ef}, $\frac{1}{\tau}$ must be an eigenvalue of $-\Delta$ on $B_R$ and  for any eigenfunction $u$ must satisfy the following  condition:
    \begin{equation*}
        \int_{\partial B_R} u(y) \frac{\partial}{\partial\eta_y} \log \frac{1}{|x-y|} d\sigma_y-\int_{\partial B_R} \frac{\partial u}{\partial\eta_y} \log \frac{1}{|x-y|}d\sigma_y=0.
    \end{equation*}
    Since $\frac{1}{\tau}$ is an eigenvalue of $-\De$ on $B_R$, there must exist  $m\in\N_0$ and $\Phi$ on $(0,R)$ such that $\Phi$ satisfies \eqref{slproblemradial}. In particular, any eigenfunction of $-\De$ corresponding to $\frac{1}{\tau}$ is a linear combination of  $\Phi(r)\cos(m\theta)$ and $\Phi(r)\sin (m\theta)$. For an eigenfunction $u$ and $x\in B_R$, let  $$A=\int_{\partial B_R} u(y) \frac{\partial}{\partial\eta_y} \log \frac{1}{|x-y|} d\sigma_y,\text{ and }B=\int_{\partial B_R} \frac{\partial u}{\partial\eta_y} \log \frac{1}{|x-y|}d\sigma_y.$$  
    For the ease of computation, first we calculate $A-B$ for $\mathcal{F}$ where  
    \begin{equation*}
        \mathcal{F}(x)=\Phi(r)e^{im\theta},\,m\in\N_0.
    \end{equation*} 
    Now for any $x=re^{i\theta}$ and $y=Re^{i\omega}$, we have
    
\begin{align}
    \frac{\partial}{\partial\eta_y} \log \frac{1}{|x-y|}&=\frac{x-y}{|x-y|^2}\cdot \frac{y}{R}   = \frac{r\cos (\omega-\theta)-R}{|re^{i\theta}-R e^{i\omega}|^2}=\frac{1}{2}\left[ \frac{e^{i\omega}}{re^{i\theta}-Re^{i\omega}} + \frac{ e^{-i\omega}}{re^{-i\theta}-R e^{-i\omega}}\right]\\
\end{align}
Next, we compute $A$:
\begin{align}
    A&= \int_{\partial B_R} \mathcal{F}(y) \frac{\partial}{\partial\eta_y} \log \frac{1}{|x-y|} d\sigma_y
    =\frac{\Phi(R)}{2 } \int_0^{2\pi}   \left[ \frac{e^{i(m+1)\omega}}{re^{i\theta}-Re^{i\omega}} + \frac{ e^{i(m-1)\omega}}{re^{-i\theta}-R e^{-i\omega}}\right] d\omega
\end{align}
By  applying the Residue theorem and Cauchy Integral formula, we obtain
\begin{align}
    \int_0^{2\pi}  \frac{e^{i(m+1)\omega}}{re^{i\theta}-Re^{i\omega}} d\omega&=-\frac{1}{Ri}\int_{|\zeta|=1}\frac{\zeta^m}{\zeta-\frac{re^{i\theta}}{R}}d\zeta
    =-\frac{2\pi}{R}   \left(\frac{re^{i\theta}}{R}\right)^m    .
\end{align}
 Similarly, 
\begin{align}
    \int_0^{2\pi} \frac{ e^{i(m-1)\omega}}{re^{-i\theta}-R e^{-i\omega}} d\omega&= \begin{cases}
       -\frac{2\pi}{R}   & \text{ if } m=0,\\
     0       & \text{ if } m\in\N
    \end{cases}
\end{align}
Combining the above integrals, we get
\begin{align}\label{I2}
    A=-\frac{\pi\Phi(R)}{R} \begin{cases}
        2 & \text{ if } m= 0,\\
       \left(\frac{re^{i\theta}}{R}\right)^m  & \text{ if } m\in\N.
    \end{cases}
\end{align}
Observe that, 
\begin{align*}
    \log \frac{1}{|x-y|}&= \log \frac{1}{|re^{i\theta}-Re^{i\omega}|}=\frac{1}{2}\left[ \log \frac{1}{re^{i\theta}-Re^{i\omega}}+ \log \frac{1}{re^{-i\theta}-Re^{-i\omega}} \right]\\
    &=\frac{1}{2}\left[2\log \frac{1}{R}+\log \frac{1}{1-\frac{re^{-i\theta}}{R}e^{i\omega}}+\log \frac{1}{1-\frac{re^{i\theta}}{R}e^{-i\omega}} \right].
\end{align*}
Therefore,
\begin{align}
    B&=\int_{\partial B_R} \frac{\partial \mathcal{F}}{\partial\eta_y} \log \frac{1}{|x-y|}d\sigma_y\\
    &=  \frac{\Phi'(R)}{2}\int_0^{2\pi} e^{ im\omega}\left[2\log \frac{1}{R}+\log \frac{1}{1-\frac{re^{-i\theta}}{R}e^{i\omega}}+\log \frac{1}{1-\frac{re^{i\theta}}{R}e^{-i\omega}} \right]d\omega.
\end{align}
We have
\begin{align}
    \int_0^{2\pi} e^{ im\omega}2\log \frac{1}{R}d\omega&=\begin{cases}
        4\pi \log \frac{1}{R}&  m=0,\\
        0& m\in\N
    \end{cases}
\end{align}
Noting that, 
\begin{equation*}
    \log\frac{1}{1-z}=z+\frac{z^2}{2}+\frac{z^3}{3}+\cdots,
\end{equation*}
we easily obtain
\begin{align}
    \int_0^{2\pi} e^{ im\omega} \log \frac{1}{1-\frac{re^{-i\theta}}{R}e^{i\omega}}d\omega&=0,\text{ for all } m\in\N_0,
\end{align}
and
\begin{align}
    \int_0^{2\pi} e^{ im\omega} \log \frac{1}{1-\frac{re^{i\theta}}{R}e^{-i\omega}}d\omega=\begin{cases}
     \frac{2\pi}{m}  \left(\frac{re^{i\theta}}{R}\right)^m  & \text{ if } m\in\N, \\
     0       & \text{ if } m=0.\end{cases}
\end{align}
Therefore, we obtain
\begin{align}\label{I3}
    B=\pi\Phi'(R)\begin{cases}
        2 \log \frac{1}{R} & \text{ if } m= 0,\\
      \frac{1}{m}\left(\frac{re^{i\theta}}{R}\right)^m & \text{ if } m\in\N.
    \end{cases}
\end{align}
For each $x\in B_R$, combining $\eqref{I2}$, and $\eqref{I3}$, we get
\begin{equation}\label{A_B_compute}
    A-B=\begin{cases}
             -\frac{2\pi}{R} [\Phi(R)-R\log R\Phi'(R)], & m=0,\\
             -\frac{\pi}{mR}\left(\frac{re^{i\theta}}{R}\right)^m [m\Phi(R)+R\Phi'(R)], &m\in\N.
         \end{cases}
\end{equation}
Notice that $\Phi(r)\cos (m\theta)=Re\, \mathcal{F}$ is also  an eigenfunction of $\Lom$ corresponding to $\tau$ and hence satisfies \eqref{ef_bdry}. Consequently, $Re \, (A-B)=0$ and by taking $x\neq 0$, we obtain two transcendental equations as follows:
\begin{equation}
    \Phi(R)-R\log R \Phi'(R)=0,\quad \text{ for } m=0,
\end{equation}
and 
\begin{equation}
    m\Phi(R)+R\Phi'(R)=0.\quad \text{ for }m\in\N.
\end{equation}

Conversely, assume that $\Phi$ satisfies $\eqref{slproblemradial}$ for some $\tau$ and $m$ and also satisfies the associated boundary condition $\eqref{m=0}$ or $\eqref{m-1onwards}$. Then, it is clear that, for $u_m(r,\theta)=\Phi(r)\cos (m\theta)$ satisfies \eqref{lap_u=la_u}. For $\mathcal{F}(r,\theta)=\Phi(r)e^{im\theta}$, from $\eqref{A_B_compute}$, together with $\eqref{m=0}$ or $\eqref{m-1onwards}$ we have $Re \, (A-B)=0$. Therefore $u_m(r,\theta)$ satisfies \eqref{ef_bdry}. Now, the proof follows from Lemma \ref{lapla_ef}. 
    
\end{proof}

\begin{remark}\label{complete_eigensystem}
   It can be observed that,  $u_{m,n}(x)=u_{m,n}(r,\theta)=\Phi_{m,n}(r)\cos (m\theta),m\in\N_0$ and  similarly, $v_{m,n}(x)=v_{m,n}(r,\theta)=\Phi_{m,n}(r)\sin (m\theta),m\in\N$ are eigenfunctions of $\Lom$ on $B_R$.  To conclude that these are the only eigenfunctions, it suffices to show that these functions form a complete orthogonal system for $L^2(B_R)$. We adopt a similar approach to the one outlined in \cite[p.56]{methods_of_math_physics}. The functions $\{1,\cos(m\theta),\sin(m\theta)\}$ forms an orthogonal basis for $L^2(0,2\pi)$. For $m=0$, the functions $\Phi_{0,n}(r)$ satisfying the Sturm-Liouville eigenvalue problem on $(0,R)$:
\begin{align*}
    \frac{d}{dr}(r\Phi')+\frac{r}{\tau} \Phi=0,\\
     \Phi(R)-R\ln R \Phi'(R)=0.
\end{align*}
For $m\neq 0$, the functions  $\Phi_{m,n}(r)$ satisfying 
\begin{align*}
    \frac{d}{dr}(r\Phi')+\left(\frac{r}{\tau}-\frac{m^2}{r} \right)\Phi=0,\\
    \quad m\Phi(R)+R \Phi'(R)=0.
\end{align*}
 For each $m\in\N_0$, $\Phi_{m,n}(r)$ constitute an orthogonal basis for $L^2[(0, R),rdr]$.  
\par For any $f\in L^2(B_R)$, let $g_r(\theta):=f(r,\theta)$ for a fixed $r$. As a consequence of Fubini’s theorem, $g_r\in L^2(0,2\pi)$ and we can express $f$ as follows:
\begin{equation*}
    f(x)=f(r,\theta)=\sum_{m\in\N_0} a_m(r)\cos (m\theta)+\sum_{m\in\N} b_m(r)\sin (m\theta), \quad 
\end{equation*}
where 
\begin{equation*}
    a_m(r)=\int_0^{2\pi} g_r(\theta)\cos (m\theta) d\theta\text{ and }
    b_m(r)=\int_0^{2\pi} g_r(\theta)\sin (m\theta) d\theta.
\end{equation*}
 Since $f\in L^2(B_R)$, from Fubini's theorem we have $a_m(r),b_m(r)\in L^2(0,R)$ , we can write $$a_m(r)=\sum_{n\in\N} c_{m,n} \Phi_{m,n}(r) .$$ Therefore, the products $\{\Phi_{m,n}(r)\cos(m\theta):m\in \N_0\}\cup\{\Phi_{m,n}(r)\sin(m\theta):m\in\N\} $  forms an orthogonal system for $L^2(B_R)$.
\end{remark}

\noindent We need the following proposition for proving Theorem \ref{radial}.
\begin{proposition}\label{bessel_prop_unique}

\begin{enumerate}[(i)]
    \item Let $I_0$ be the modified Bessel function of order zero. Then, for each  $\alpha>0$,  the function $I_0(t)-\alpha t I_0'(t)$ vanishes exactly at one point $t_\alpha \in (0,\infty)$.
    Moreover, $t_\alpha$ is monotonically decreasing.
    
    \item Let $J_0$ be the Bessel function of order zero. Then, for each  $\alpha\in \R$,  the function $J_0(t)-\alpha t J_0'(t)$ vanishes exactly once between any two zeroes of $J_0$.
\end{enumerate}
    
\end{proposition}
\begin{proof}
 (i) Define 
 \begin{equation}\label{appendix_g(t)}
     g(t)=\frac{tI_0'(t)}{I_0(t)}.
 \end{equation}
  Clearly, the function $I_0(t)-\alpha t I_0'(t)$ vanishes at a point $t_\alpha$ if and only if $\frac{1}{\alpha}=g(t_\alpha)$. Therefore, the proof will be complete if we show that $g$ is a bijection from 
    $[0,\infty)$ to $[0,\infty).$ 
    For this, we compute 
    \begin{align*}
        g'(t)&= \frac{I_0(t)[tI_0'(t)]'-tI_0'(t)^2}{I_0(t)^2}
        =\frac{t(I_0(t)^2-I_0'(t)^2)}{I_0(t)^2}>\frac{t(I_0(t)-I_0'(t))}{I_0(t)^2},
    \end{align*}
    where the equality follows as $$[tI_0'(t)]'-tI_0(t)=tI_0''(t)+I_0'(t)-tI_0(t)=0.$$ By considering the integral expression $I_0(t)=\frac{1}{\pi} \int_{-1}^{1} \frac{e^{-t\theta}}{\sqrt{1-\theta^2}}d\theta$, see \cite[p.237]{nico_temme_bessel}, we obtain
      \begin{equation}\label{i_0-I_1}
        I_0(t)-I_0'(t)=\frac{1}{\pi} \int_{-1}^{1} \frac{(1+\theta)e^{-t\theta}}{\sqrt{1-\theta^2}}d\theta>0.
    \end{equation}
    Thus, $g$ is strictly increasing. Moreover, using the L’ Hopital’s rule and \eqref{i_0-I_1}, we obtain
    \begin{align*}
        \lim_{t\rightarrow\infty} g(t)&=\lim_{t\rightarrow\infty}\frac{tI_0'(t)}{I_0(t)}=\lim_{t\rightarrow\infty} \frac{tI_0(t)}{I_0'(t)}
        \geq\lim_{t\rightarrow\infty}t.
    \end{align*}
        Hence, $g(t)\rightarrow\infty$ as $t\rightarrow\infty$. Thus, $g$ must be a bijection from 
    $[0,\infty)$ to $[0,\infty).$ This can also be observed from Fig \ref{graphofh}.
    Let $\alpha_1<\alpha_2$. If $t_{\alpha_2}\geq t_{\alpha_1}$, then  since $g$ is increasing, we must have $g(t_{\alpha_2})\geq g(t_{\alpha_1})$ leads to the contradiction $\alpha_1\geq\alpha_2$. Thus, $t_{\alpha_2}< t_{\alpha_1}$ and hence $t_\alpha$ is monotonically decreasing.   This completes the proof.

 \vspace{5mm}
 
\noindent (ii)Let $h(t)=\frac{tJ_0'(t)}{J_0(t)}=-\frac{tJ_1(t)}{J_0(t)}$ for $t>0$ such that $J_0(t)\neq 0$. Then, 
    \begin{equation}
        h'(t)=-\frac{t\left[ J_0(t)^2+J_1(t)^2 \right]}{J_0(t)^2}<0,\text{ when } J_0(t)\neq 0.
    \end{equation}
    Consequently, $h$ is strictly decreasing from 0 to $-\infty$ between  $(0,j_{0,1})$ and strictly decreasing from $\infty$ to $-\infty$ between any two zeroes of $J_0$ as illustrated in Fig \ref{graphofg}.
    
\vspace{5mm}

\begin{minipage}[t]{0.47\textwidth} 
             \begin{center}
            \captionsetup{type=figure}
            \includegraphics[scale=0.5]{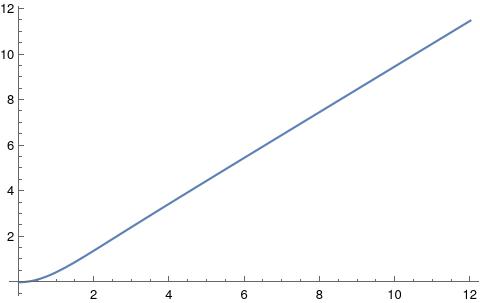}
            \captionof{figure}{Graph of $g$}\label{graphofh}
            \end{center}
     \end{minipage}    
    \begin{minipage}[t]{0.47\textwidth} 
             \begin{center}
            \captionsetup{type=figure}
            \includegraphics[scale=0.5]{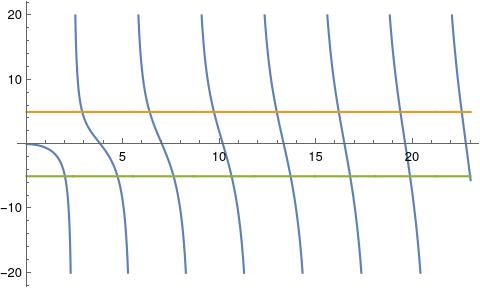}
            \captionof{figure}{Graph of $h$}\label{graphofg}
            \end{center}
     \end{minipage}

\end{proof}

\noindent Now we will prove Theorem \ref{radial}.\\
\textbf{Proof of Theorem \ref{radial}:}  Let $u$ be a radial eigenfunction corresponding to a positive eigenvalue $\tau$ of $\Lom$. Then, $u(x)=\Phi(|x|)$, where $\Phi$ satisfies the Bessel equation($m=0$ case), $$r^2\Phi''+r\Phi'+\la r^2\Phi=0,\quad (\la=\frac{1}{\tau})$$  with the boundary condition,  $$ \Phi(R)-R\log R\,\Phi'(R)=0.$$  Thus, $\Phi(r)=J_0(\sqrt{\la}r)$ and     
\begin{equation}\label{bessel_boundary}
    J_0(\sqrt{\la}R)-\sqrt{\la}R\log R J_0'(\sqrt{\la}R)=0.
\end{equation}
This equation has infinitely many roots follows from Proposition \ref{bessel_prop_unique}. Consequently,
\begin{equation}
        \tau_{0,n}(B_R)=\frac{R^2}{\mu_{0,n}(B_R)^2}\text{ for some }n\in\N,
\end{equation}
where $\mu_{0,n}(B_R)$ is the $n^{th}$ root of the  equation $J_0(t)-\log R\, t J_0'(t)=0$.  The corresponding eigenfunctions are given by 
   \begin{equation*}
        u_{0,n}(x)=J_0\left(\frac{\mu_{0,n}(B_R)}{R}|x|\right),\quad n\in\N.
    \end{equation*}    
For $R\leq 1$, all the eigenvalues are positive and hence above $\tau_{0,n}(B_R)$ gives all the radial eigenvalues. For $R>1$, it has already been proved in Theorem \ref{radial_small_version}.

\qed


\begin{theorem}\label{nonradial}
   Let $B_R$ be the disc with radius $R$. The non-radial eigenvalues $\tau_{m,n}$ on $B_R$ are given by 
   \begin{equation*}
        \tau_{m,n}(B_R)=\frac{R^2}{j_{m-1,n}^2}\text{ for }m\in\N,
    \end{equation*}
    where $j_{m,n}$ denotes the $n^\text{th}$ positive zero of the Bessel  function $J_m$. The corresponding eigenfunctions(up to a constant multiple) are given by
    \begin{equation*}
        u_{m,n}(x)=u_{m,n}(r,\theta)=J_m\left(\frac{j_{m-1,n}}{R}r\right)\cos (m\theta),\quad m, n\in\N,
    \end{equation*} 
    and
    \begin{equation*}
        v_{m,n}(x)=v_{m,n}(r,\theta)=J_m\left(\frac{j_{m-1,n}}{R}r\right)\sin (m\theta),\quad m, n\in\N.
    \end{equation*}
\end{theorem}

\begin{proof}
        Let $\tau$ be a non-radial eigenvalue of $\Lom$ and  $u$ be the corresponding eigenfunction. It follows that $m\neq 0$ by Remark \ref{complete_eigensystem}. Then, by Lemma \ref{ball_lemma1}, there exists an $m\in\N$ and $\Phi$ satisfying the Bessel equation of order m, $$r^2\Phi''+r\Phi'+(\la r^2-m^2)\Phi=0,\quad \left(\la=\frac{1}{\tau}\right)$$  with the boundary condition: 
        \begin{equation}\label{m1onwards}
            m\Phi(R)+R\Phi'(R)=0.
        \end{equation}
        Consequently,
        \begin{equation}\label{phiformnonrad}
            u_m(x)=u_m(r,\theta)=J_m(\sqrt{\la}r)\cos(m\theta), m\in \N,
        \end{equation}
        and 
        \begin{equation}
            v_m(x)=v_m(r,\theta)=J_m(\sqrt{\la}r)\sin(m\theta), m\in \N.
        \end{equation}
         From $\eqref{m1onwards}$ and by using the recurrence relation of Bessel functions\cite[formula (3)-p.18]{watson_Bessel} we get 
         \begin{equation}
             mJ_{m}(\sqrt{\la}R)+\sqrt{\la} RJ_{m}'(\sqrt{\la}R)=\sqrt{\la}R J_{m-1}(\sqrt{\la} R)=0.
         \end{equation}
        Moreover, $J_{m-1}(\sqrt{\la} R)=0$. Hence, for each $m\in\N$,    
        \begin{equation*}
                \tau_{m,n}(B_R)=\frac{R^2}{j_{m-1,n}^2},\,\; n\in\N,
        \end{equation*}
        where $j_{m,n}$ denotes the $n^\text{th}$ positive zero of the Bessel  function $J_m$. The corresponding eigenfunctions are given by 
         \begin{equation*}
        u_{m,n}(x)=u_{m,n}(r,\theta)=J_m\left(\frac{j_{m-1,n}}{R}r\right)\cos (m\theta),\quad m, n\in\N,
    \end{equation*} 
    and
    \begin{equation*}
        v_{m,n}(x)=v_{m,n}(r,\theta)=J_m\left(\frac{j_{m-1,n}}{R}r\right)\sin (m\theta),\quad m, n\in\N.
    \end{equation*}

 \end{proof}


\begin{theorem}\label{summary}
    Let $B_R$ be a disc of radius $R$.
     \begin{enumerate}[(I)]
         \item If $R=1$, then $\tau_{0,n}(B_1)=\frac{1}{j_{0,n}^2},\,\; n\in\N$.
         \item If $R<1$, then $$ \frac{R^2}{j_{0,n}^2}<\tau_{0,n}(B_R)<\frac{R^2}{j_{0,n-1}^2} ,\text{ for }n\geq 2,$$ and
         $$\frac{R^2}{j_{0,1}^2}<\tau_{0,1}(B_R).  $$  Furthermore,
            
            \begin{equation}
                \tau_{0,1}(B_R)\approx\frac{R^2|\log\, R|}{2}  \text{ when $R$ is near } 0.
            \end{equation}
            \item If $R>1$, then $$\frac{R^2}{j_{0,n+1}^2}<\tau_{0,n}(B_R)<\frac{R^2}{j_{0,n}^2},\,\; n\in\N.$$
            Moreover,  
        \begin{equation*}
        \Tilde{\tau}_1(B_R)\approx\begin{cases} -\frac{R^2 \,\log\, R}{2} 
                    & \text{ when  $R$ is near } 1,\\
            -R^2(\log\, R)^2 & \text{ when $R$ is near $\infty$} .
        \end{cases}
    \end{equation*}
         
            
        

    \item The non-radial eigenvalues $\tau_{m,n}(B_R)=R^2\tau_{m,n}(B_1),\,\,m,n\in\N$.
    \end{enumerate}
\end{theorem}
            
        

\begin{proof}
\noindent (I) From Theorem \ref{radial}, for $R=1$, the radial eigenvalues are $\tau_{0,n}(B_1)=\frac{1}{j_{0,n}^2},\,\;n\in\N$.

\vspace{5 mm}

\par  For an arbitrary  $R$, the positive radial eigenvalues are $\tau_{0,n}(B_R)=\frac{R^2}{(\mu_{0,n}(B_R))^2}$ where $\mu_{0,n}(B_R)$ are the $n^\text{th}$ root of the equation $J_0(t)-\log R t J_0'(t)=0\,(\text{or }J_0(t)+\log R t J_1(t)=0)$. Recall,
        \begin{equation*}
            h(t)=-\frac{t J_1(t)}{J_0(t)}, \,\, \text{for }t\text{ such that }J_0(t)\neq 0.
        \end{equation*}
     \noindent As mentioned in Proposition \ref{bessel_prop_unique}, $h$ is strictly decreasing from $+\infty$ to $-\infty$ between any two zeroes of $J_0$(See Fig.\ref{graphofg}). We have $\mu_{0,n}(B_R)$ satisfies, $$h(\mu_{0,n}(B_R))=-\frac{\mu_{0,n}(B_R)J_1(\mu_{0,n}(B_R))}{J_0(\mu_{0,n}(B_R))}=\frac{1}{\log\, R}.$$ Thus $\mu_{0,n}(B_R)$ are precisely the $n^\text{th}$ points where the horizontal line  $\frac{1}{\log\, R}$ intersects the graph of $h$(See Fig.\ref{graphofg}).

     \vspace{5mm}

    \noindent (II) For $R<1$, $0<\mu_{0,1}(B_R)<j_{0,1}$ and $j_{0,n-1}<\mu_{0,n}(B_R)<j_{0,n},n\geq 2$. Since $h(t)=-\frac{t J_1(t)}{J_0(t)}\approx -\frac{t^2}{2}$ when $t\text{ is near to } 0$, we have, $\mu_{0,1}(B_R)\approx \sqrt{\frac{2}{|\log\, R|}}$ when $R$ is near to $0$.

             \begin{center}
            \captionsetup{type=figure}
            \includegraphics[scale=0.5]{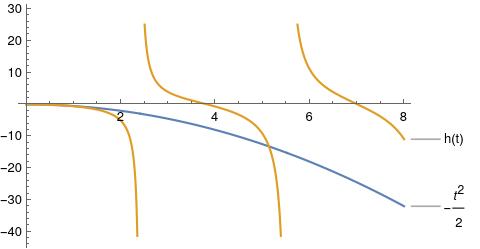}
            \captionof{figure}{Graph of $h$ and $-\frac{t^2}{2}$}\label{handt2by2}
            \end{center}

     \vspace{5 mm}
     
\noindent (III) For $R>1$, from the previous arguments we have $j_{0,n}<\mu_{0,n}(B_R)<j_{0,n+1},\,\; n\in\N$. The asymptotic behaviour of the unique negative eigenvalue $\Tilde{\tau}_1(B_R)$ is established in the proof of Theorem \ref{asymptotic_small_version}.


    \vspace{5 mm}

 \noindent (IV) Follows directly from Theorem \ref{nonradial}.
\end{proof}

\section{First, Second, and Third eigenvalues on \texorpdfstring{$B_R$}{BR}}\label{AppendixB}


In this section, we will examine the multiplicity and radiality of the first three largest eigenvalues $\tau_1(B_R)\geq\tau_2(B_R)\geq \tau_3(B_R)$ of $\Lom$ on $B_R$. The positive eigenvalues on $B_R$ can be enumerated by ordering  $\mu_{0,n},\,n\in\N$ and $j_{m,n},\,n\in\N$. 
\begin{center}
    $\begin{matrix}
    \mu_{0,1}(B_R) & < &  \mu_{0,2}(B_R) &< & \mu_{0,3}(B_R)&< &\cdots&\text{    (depends on } R)\\
    \hfill \\
    j_{0,1}&< & j_{0,2}&<& j_{0,3}&<&\cdots \\
    \wedge & & \wedge& & \wedge\\
    j_{1,1}& < & j_{1,2}& <& j_{1,3}& <&\cdots \\
    \wedge& & \wedge& & \wedge\\
    \vdots & &\vdots & &\vdots & &\ddots 
    \end{matrix}    $
\end{center}
Now, recall the function $h(t)=  -\frac{tJ_1(t)}{J_0(t)}$. Notice that, $h$ is negative in $(0\, ,\,j_{0,1})$ and strictly decreasing from $+\infty$ to $-\infty$ in $(j_{0,1}\, ,\, j_{0,2})$. 
 \begin{center}
            \captionsetup{type=figure}
            \includegraphics[scale=0.6]{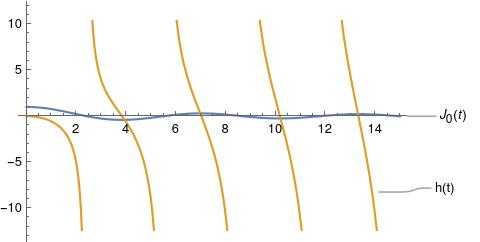}
            \captionof{figure}{Graph of $J_0$ and $h$}\label{besselgraphmultiplicity}
\end{center}

\noindent\underline{\textbf{On a disc with $R<1$:}}  We have, $$h(\mu_{0,1}(B_R))=\frac{1}{\log\, R}<0.$$ 
      Consequently, $\mu_{0,1}(B_R)<j_{0,1}$. Thus, we obtain the following ordering:
      \begin{equation*}
    \begin{matrix}
    \mu_{0,1}(B_R) & < & \fbox{$ \mu_{0,2}(B_R)$} &< & \mu_{0,3}(B_R)&< &\cdots&\text{    (depends on } R)\\
    \wedge& & \wedge& & \wedge\\
    \fbox{$j_{0,1}$}&< & j_{0,2}&<& j_{0,3}&<&\cdots \\
    \wedge & & \wedge& & \wedge\\
    j_{1,1}& < & j_{1,2}& <& j_{1,3}& <&\cdots \\
    \wedge& & \wedge& & \wedge\\
    \vdots & &\vdots & &\vdots & &\ddots 
    \end{matrix}    
    \end{equation*}
      Therefore,      
      $$\tau_{1}(B_R)=\tau_{0,1}(B_R)=\frac{R^2}{(\mu_{0,1}(B_R))^2},$$  and from Theorem \ref{radial}, $\tau_1(B_R)$ admits only radial eigenfunctios  and its multiplicity is 1 . The second eigenvalue can be determined by 
\begin{equation}\label{tau_2new}
    \tau_2(B_R)=\sup \{\tau_{m,n}(B_R): \tau_{1}(B_R)>\tau_{m,n}(B_R), m\in \N_0, n\in\N \}.
\end{equation}
In the proof of Theorem \ref{summary}(II), we mentioned that $0<\mu_{0,1}(B_R)<j_{0,1}<\mu_{0,2}(B_R)<j_{0,2}$. Therefore,
\begin{equation}
\tau_2(B_R)=\tau_3(B_R)=\tau_{1,1}(B_R)=\frac{R^2}{j_{0,1}^2}.
\end{equation}
Thus, second and third eigenvalues admit only non-radial eigenfunctions.
\\

\noindent\underline{\textbf{On $B_1$:} }
From Theorem \ref{radial}, since $\log R =0$, it follows that $\mu_{0,n}(B_1)=j_{0,n}$. Thus, 
\begin{center}
    $\begin{matrix}
    \mu_{0,1}(B_1) & < &  \mu_{0,2}(B_1) &< & \mu_{0,3}(B_1)&< &\cdots\\
    \shortparallel & & \shortparallel & &\shortparallel \\
    j_{0,1}&< & j_{0,2}&<& j_{0,3}&<&\cdots \\
    \wedge& & \wedge& & \wedge \\
    j_{1,1}& < & j_{1,2}& <& j_{1,3}& <&\cdots \\
    \wedge& & \wedge& & \wedge\\
    \vdots & &\vdots & &\vdots & &\ddots 
    \end{matrix}    $
\end{center}
Therefore, all radial eigenvalues admit non-radial eigenfunctions and have multiplicity exactly 3\cite[Theorem 3.1]{anderson}. In particular,
     $$\tau_1(B_1)=\tau_2(B_1)=\tau_3(B_1)=\tau_{0,1}(B_1)=\tau_{1,1}(B_1)=\frac{1}{j_{0,1}^2}.$$     
\\

\noindent\underline{\textbf{On a disc with $R>1$:} } We have,
      $$h(\mu_{0,1}(B_R))=\frac{1}{\log\, R}>0.$$            
      \noindent This implies that $j_{0,1}<\mu_{0,1}(B_R)<j_{0,2}$. Therefore, 
     \begin{equation*}
    \begin{matrix}
    \fbox{$\mu_{0,1}(B_R)$} & < &  \mu_{0,2}(B_R)  &< & \mu_{0,3}(B_R)&< &\cdots&\text{    (depends on } R)\\
    \vee & & \vee& & \vee\\
    j_{0,1}&< & \fbox{$j_{0,2}$}&<& j_{0,3}&<&\cdots \\
    \wedge & & \wedge& & \wedge\\
    \fbox{$j_{1,1}$}& < & j_{1,2}& <& j_{1,3}& <&\cdots \\
    \wedge& & \wedge& & \wedge\\
    \vdots & &\vdots & &\vdots & &\ddots 
    \end{matrix}    
\end{equation*}
By Theorem \ref{nonradial}, $$\tau_1(B_R)=\tau_2(B_R)=\tau_{1,1}(B_R)=\frac{R^2}{j_{0,1}^2},$$    
and these eigenvalues admit only non-radial eigenfunctions. The third eigenvalue is given by:
\begin{equation}\label{tau_3new}
    \tau_3(B_R)=\sup \{\tau_{m,n}(B_R): \tau_{1}(B_R)>\tau_{m,n}(B_R) \}.
\end{equation}
Consider the function $f(t)=-\frac{J_0(t)}{tJ_1(t)}$. Similarly to earlier discussions, it can be shown that $f$ is strictly increasing from $-\infty$ to $\infty$ between any two zeroes of $J_1$. Notice that, $f(\mu_{0,1}(B_R))=\log\, R\, >\,0$.
\begin{center}
\captionsetup{type=figure}
    \includegraphics[scale=0.8]{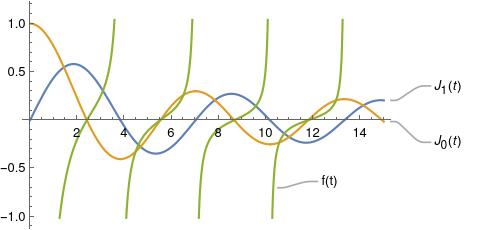}
    \captionof{figure}{Graph of $J_0,\, J_1$ and $f$}
\end{center}
\noindent It follows that $\mu_{0,1}(B_R)<j_{1,1}\approx 3.8317 < j_{0,2}\approx 5.5201$. Hence, the third eigenvalue $\tau_3(B_R)=\frac{R^2}{(\mu_{0,1}(B_R))^2}$, which admits only  radial eigenfunctions.

\bibliographystyle{abbrvurl}
\bibliography{Reference}

\end{document}